\documentclass[letterpaper,11pt]{article}

\usepackage[margin=1in]{geometry}
\usepackage{graphicx}

\usepackage{authblk}
\usepackage[linesnumberedhidden]{algorithm2e}
\usepackage[nottoc,numbib]{tocbibind}

\usepackage{enumitem}

\usepackage{tikz}
\usetikzlibrary{arrows,calc,shapes}
\usepackage{tikz-cd}

\tikzstyle{arc}=[->,shorten <=3pt, shorten >=3pt,
                 >=stealth, line width=1.1pt]
\tikzstyle{edge}=[shorten <=2pt, shorten >=2pt,
                  >=stealth, line width=1.1pt]
\tikzstyle{blueE}=[shorten <=2pt, shorten >=2pt,
                  >=stealth, line width=1.5pt, blue]
\tikzstyle{redE}=[shorten <=2pt, shorten >=2pt,
                  >=stealth, line width=1.5pt, dashed, red]
\tikzstyle{vertex}=[circle, fill=white, draw,
                    minimum size=7pt,
                    inner sep=0pt, outer sep=0pt]
\tikzstyle{redA}=[->, dashed, shorten <=3pt, shorten >=3pt, red, >=stealth,
                  line width=1.5pt]
\tikzstyle{blueA}=[->, shorten <=3pt, shorten >=3pt, blue, >=stealth,
                   line width=1.5pt]

\usepackage{float}

\usepackage{xcolor}

\usepackage[hidelinks]{hyperref} 

\usepackage{amsmath}
\usepackage{amssymb}
\usepackage{bbm} 
\usepackage{xfrac}

\usepackage{amsthm}

\usepackage{multicol}

\newtheorem{theorem}{Theorem}[section]
\newtheorem{lemma}[theorem]{Lemma}
\newtheorem{corollary}[theorem]{Corollary}
\newtheorem{proposition}[theorem]{Proposition}
\newtheorem{observation}[theorem]{Observation}

\theoremstyle{remark}
\newtheorem{remark}{Remark}

\theoremstyle{definition}
\newtheorem{example}{Example}

\newcommand{\calC}{\mathcal C}

\newcommand{\calF}{\mathcal F}

\newcommand{\calT}{\mathcal T}

\newcommand{\bA}{\mathbb A}
\newcommand{\bB}{\mathbb B}
\newcommand{\bC}{\mathbb C}
\newcommand{\bD}{\mathbb D}

\newcommand{\bG}{\mathbb G}
\newcommand{\bH}{\mathbb H}
\newcommand{\bS}{\mathbb S}
\newcommand{\bF}{\mathbb F}

\newcommand{\ignore}[1]{}

\DeclareMathOperator{\CSP}{CSP}

\DeclareMathOperator{\SNP}{SNP}
\DeclareMathOperator{\Forb}{Forb}
\DeclareMathOperator{\NP}{NP}

\renewcommand{\P}{\textnormal{P}}

\DeclareMathOperator{\coNP}{coNP}

\DeclareMathOperator{\PO}{P}

\DeclareMathOperator{\HER}{Her}
\DeclareMathOperator{\HerFO}{Her-FO}
\DeclareMathOperator{\FO}{FO}

\DeclareMathOperator{\Lin}{Lin}

\DeclareMathOperator{\Cyc}{Cyc}

\title{Hereditary First-Order Logic: \\ the tractable quantifier prefix classes\thanks{This project has received funding from the European Union
(Project POCOCOP, ERC Synergy Grant 101071674).
Views and opinions expressed are however those of the author(s) only and do not
necessarily reflect those of the European Union or the European Research
Council Executive Agency. Neither the European Union nor the granting
authority can be held responsible for them.}} 

 \author[1]{Manuel Bodirsky\thanks{manuel.bodirsky@tu-dresden.de}}
 \author[1]{Santiago Guzm\'an-Pro\thanks{santiago.guzman\_pro@tu-dresden.de}}
 \affil[1]{Institut f\"ur Algebra, TU Dresden, Germany}

\begin{document}

\maketitle

\begin{abstract}
Many computational problems can be modelled  as the class  of all finite structures $\bA$
that satisfy a fixed first-order sentence $\phi$ \emph{hereditarily}, i.e., we require that
every (induced) substructure of $\bA$ satisfies $\phi$. We call the corresponding
computational problem the \emph{hereditary model checking} problem for $\phi$, and
denote it by $\HER(\phi)$.

We present a complete description of the quantifier prefixes for $\phi$ such that
$\HER(\phi)$ is in P; we show that for every other quantifier prefix there exists a formula
$\phi$ with this prefix such that $\HER(\phi)$ is coNP-complete.
Specifically, we show that if $Q$ is of the form $\forall^\ast\exists\forall^\ast$
or of the form $\forall^\ast\exists^\ast$, then $\HER(\phi)$ can be solved in polynomial time
whenever the quantifier prefix of $\phi$ is $Q$. Otherwise, $Q$ contains $\exists \exists \forall$ 
or $\exists \forall \exists$ as a subword, and in this case, there is
a first-order formula $\phi$ whose quantifier prefix is $Q$ and $\HER(\phi)$ is coNP-complete.
Moreover, we show that there is no algorithm that decides for a given first-order formula $\phi$
whether $\HER(\phi)$ is in P (unless P$=$NP).

\end{abstract}

\tableofcontents

\newpage


\setcounter{page}{1}

\section{Introduction} 
\emph{Vertex-deletion problems} were first studied as optimization problems~\cite{krishamoorthySICOMP8,lewisJCSS20}:
Given an input graph $\bG$ determine the minimum number of vertices to be deleted so that the remaining (induced) subgraph
belongs to a specified class $\mathcal C$.  Krishnamoorthy and Deo~\cite{krishamoorthySICOMP8}
showed that this problem is NP-hard for several natural graph classes $\mathcal C$ such as trees,
planar, bipartite, hamiltonian, interval, and chordal graphs. Shortly after,
Lewis and Yanakkakis~\cite{lewisJCSS20} proved that  if $\mathcal C$ is a non-trivial hereditary
class of finite graphs (i.e., $\mathcal C$ is an infinite class of finite graphs which is closed under induced subgraphs and which is not
the class of all finite graphs), then the vertex-deletion problem defined
by $\mathcal C$ is $\NP$-hard.

Given the previous hardness results, it was natural to study vertex-deletion problems
from the viewpoint of approximation or of parametrized complexity. Regarding the former,
Fujito~\cite{fujitoDAM86} proposed a unified  polynomial-time  algorithm that approximates
an optimal solution to a vertex-deletion problem defined by a non-trivial hereditary class.
Some well-known results include the fixed parameter tractability of the odd-cycle transversal
problem~\cite{reedORL32}, and of the feedback vertex set problem~\cite{ramanACMTA2} ---
the former corresponds to vertex-deletion problem for bipartite graphs, and the
latter to the vertex-deletion problem for forests.

Recently, the parametrized complexity perspective has been studied systematically for
vertex-deletion problems defined by graph classes expressible by some first-order
sentence $\phi$~\cite{banachMFCS24,fominTCS64}. 
In this paper we are interested in 
the class of graphs $\bG$ such that no matter how many vertices are removed from $\bG$, 
the remaining induced subgraph does not satisfy $\phi$. Since first-order logic
is closed under negations, we prefer the positive phrasing of this problem:
Given a graph $\bG$, test whether all non-empty induced subgraphs of $\bG$ satisfy
a specified first-order formula $\psi$.

\subsection*{Hereditary first-order logic}
From now on we work in the setting of relational structures, for which graphs and digraphs are
prototypical examples. 
We use the convention that all structures have a non-empty domain.
We begin by introducing \textit{hereditary first-order logic}.
A structure $\bA$ \emph{hereditarily satisfies $\phi$} if every substructure $\bB$ of $\bA$ satisfies $\phi$. 
We denote by $\HER(\phi)$ the class of finite structures that hereditarily satisfy $\phi$. The
\textit{hereditary model checking problem} for $\phi$ consists of deciding whether an input
structure $\bA$ belongs to $\HER(\phi)$.  Since verifying whether a finite structure $\bA$
models a fixed first-order formula $\phi$ can be done in polynomial time, it follows that $\HER(\phi)$ is
in coNP. 

We say that a class $\calC$ of finite $\tau$-structures is \textit{hereditarily first-order
definable} if there is a first-order sentence $\phi$ such that $\calC = \HER(\phi)$. 
In this case, we also say that $\calC$ is in $\HerFO$.
Clearly, every class in $\HerFO$ is \emph{hereditary}, i.e., closed under taking substructures. 

\subsubsection*{Three simple examples}
\label{sect:three-expl}
Clearly, every hereditary class in $\FO$ is also in $\HerFO$. The following graph and digraph classes are 
 in $\HerFO$, but not in $\FO$.

\begin{example}[Forests]\label{ex:forests}
    Consider a first-order formula $\phi$ stating that ``there is a loopless vertex of degree at most $1$''.
    Clearly, every forest hereditarily satisfies $\phi$. Conversely, suppose that $\bG$ is a finite graph that
    hereditarily satisfies $\phi$. Then $\bG$ has a loopless vertex $v$ of degree at most $1$. We inductively see
    that the subgraph with vertex set $(G\setminus\{v\})$ is a forest, i.e., has no cycles, and since $v$ has
    degree at most $1$ and does not have a loop, we conclude that $\bG$ is a forest.\footnote{This example
    naturally generalizes to \emph{$k$-degenerate} graphs, i.e.,
    the class of graphs $\bG$ such that every subgraph of $\bG$ contains a vertex of degree at most $k$.}
\end{example}

\begin{example}[Chordal graphs]\label{ex:chordal}
    A graph $\bG$ is \emph{chordal} if every cycle of $\bG$ contains a chord; equivalently, if 
    $\bG$ contains no induced cycle of length $n\ge 4$. Rose~\cite{roseJMAA32} proved that a graph $\bG$ is
    chordal if and only if every induced
    subgraph contains a vertex whose neighbourhood induces a clique. Hence, if $\phi$ is a first-order
    sentence stating ``there is a (loopless) vertex  $v$ such that every two neighbours of $v$ are
    adjacent'',  then $\HER(\phi)$ describes the class of chordal graphs.
    It is known that membership in this class can be decided in polynomial time~\cite{RoseTajanLuecker}.
\end{example}

\begin{example}[Directed acyclic digraphs]\label{ex:DAG}
    The class of acyclic digraphs belongs to $\HerFO$ as well.  Indeed, it suffices to consider
    the first-order sentence $\exists x\forall y. \lnot E(x,y)$, i.e., the first order sentence that
    states that there exists a \emph{sink}, i.e., a vertex without outgoing edges. Equivalently,
    the constraint satisfaction problem 
    for $(\mathbb Q, <)$ (denoted by CSP$(\mathbb Q, <)$, see Section~\ref{sect:prelims}) 
    is in $\HerFO$. Also this computational problem can be solved in polynomial time (e.g., by depth-first-search).
\end{example}

The fact that these properties are not expressible in FO is well-known, and can for example be shown
using Ehrenfeucht-Fra\"i{ss}\'e games; see~\cite{EbbinghausFlum}.

$\HerFO$ is a particularly natural formalism when it comes to the description of constraint
satisfaction problems. In Section~\ref{sec:CSP} we will give natural syntactic restrictions
of first-order sentences $\phi$ that imply that $\HER(\phi)$ describes a CSP, as well
as a variety of examples of CSPs that are expressible in $\HerFO$.
An example of a coNP-complete CSP in $\HerFO$ can be found in Theorem~\ref{thm:Forb-TD} and in
Corollary~\ref{thm:HensonCSP}.

We are interested in complexity classification for the hereditary model-checking problem
for $\phi$. Some complexity results can be derived from general results by viewing 
$\HerFO$ as a fragment of universal second-order logic; we review these results in the next section.  

\subsection*{Prefix classifications}
It is straightforward to observe that $\HER(\phi)$ can be expressed in
universal monadic second-order logic.

\begin{observation}\label{obs:HerFO-UMSO}
   Consider a first-order $\tau$-formula $\phi:=Q_1x_1\dots Q_nx_n. \psi(x_1,\dots, x_n)$ and
   a finite structure $\bA$. If $S$ is a unary predicate not in $\tau$, then $\bA$  hereditarily
   satisfies $\phi$ if and only if $\bA$ models
   \[
   \forall S \; Q_1x_1\dots Q_nx_n \left (\exists z. S(z)\land \bigwedge_{i\in U}S(x_i) \implies \bigwedge_{i\in E}S(x_i)\land \psi(x_1,\dots, x_n) \right) ,
   \]
   where $U$ (respectively, $E$) is the set of indices $i\in \{1,\dots,n\}$ such that $Q_i$ is a universal (respectively, existential) quantifier.
\end{observation}

The question whether a given existential second-order (ESO)  sentence describes a polynomial-time solvable problem is easily seen to
be undecidable (see, e.g., Theorem 1.4.2 in~\cite{Book}). This motivates  quantifier-prefix dichotomy
results for existential second-order logic~\cite{Eiter2010,GottlobKolaitisSchwentick}. For instance,
the classification for monadic ESO is as follows (see Figure 1 (a) in~\cite{GottlobKolaitisSchwentick}
for the full classification). For every quantifier prefix $Q\in \{\exists,\forall\}^\ast$ the following holds: 
\begin{itemize}
    \item either $Q$ is of the form $\exists^\ast\forall$, and in this case every monadic ESO  sentence $\Phi$ whose first-order part has
    a quantifier-free prefix $Q$  
    is polynomial-time decidable, or
    \item $Q$ contains $\forall\forall$ or $\forall\exists$ as subwords, and in this case  there is an ESO sentence $\Phi$ whose
    first-order part has quantifier prefix $Q$ and deciding $\Phi$ is NP-complete.
\end{itemize}

By taking complements, it follows that $\HER(\phi)$ is in P whenever  $\phi$ is a first-order sentence with a quantifier prefix of the form $\forall^* \exists$. However, since $\HerFO$ is more restrictive than universal second-order logic, it does not follow that $\HerFO$ contains coNP-complete 
problems for first-order formulas with quantifier prefix $\exists \exists$ or $\exists \forall$. In fact, our results imply that if the quantifier prefix of $\phi$ equals 
$\exists \exists$ or $\exists \forall$, then $\HER(\phi)$ is in P.

\subsection*{Our contributions}
Our first result shows that it is undecidable to test for a given first-order formula $\phi$ whether $\HER(\phi)$ can
be solved in polynomial time. We thus follow the approach described above and consider fragments of
HerFO defined by restricting the quantifier prefix on the first-order sentence.
We present the following classification of the computational complexity of problems in $\HerFO$ based
on the quantifier prefix $Q \in \{\exists,\forall\}^*$ of the fixed first-order formula: 
\begin{itemize}
    \item either $Q$ is of the form $\forall^* \exists^*$ or the form $\forall^* \exists \forall^*$, and in this case $\HER(\phi)$ is
    polynomial-time decidable for every first-order formula $\phi$ with quantifier prefix $Q$, or
    \item  $Q$ contains $\exists\exists\forall$ or  $\exists \forall \exists$ as a subword, and in this case 
    there are first-order formulas with quantifier prefix $Q$ such that $\HER(\phi)$ is $\coNP$-complete.
\end{itemize}

\subsection*{Outline} 
We begin by recalling some basic concepts from finite model theory (Section~\ref{sect:prelims}). 
We then list further examples of problems in HerFO, and prove that certain problems cannot be expressed in $\HerFO$  (Section~\ref{sec:CSP});
it will also become clear why hereditary model checking is quite natural in the context of constraint satisfaction.
The undecidability of tractability of $\HerFO$ can be found in Section~\ref{sect:undec}.
The classification of the complexity of $\HerFO$ depending on the allowed quantifier
prefix (Theorem~\ref{thm:Q-classification}) is the main result of Section~\ref{sect:prefixes}. 
This section also contains the mentioned quantifier-prefix connection between the decidability of the meta problem for $\HerFO$ and the decidability of satisfiability in the finite.
Moreover, we specialize our results to the class of constraint satisfaction problems.
A series
of problems that are left open can be found in Section~\ref{sect:open}.

\section{Preliminaries} 
\label{sect:prelims} 
We assume basic familiarity with first-order logic and we follow standard notation from
model theory, as, e.g., in~\cite{Hodges}. We also use standard notions from complexity theory.

\subsection*{(First-order) structures}
\label{sect:structs}
Given a relational signature $\tau$ and a $\tau$-structure $\bA$,
we denote by $R^\bA$ the interpretation in $\bA$ of a relation symbol $R\in \tau$. Also, 
we denote relational structures with letters $\bA, \bB, \bC,\dots$, and their domains by
$A,B,C,\dots$. In this article, structures have non-empty domains. 

If $\bA$ and $\bB$ are $\tau$-structures,
then a \emph{homomorphism} from $\bA$ to $\bB$ is a map $f \colon A \to B$ such that
for all $a_1,\dots,a_k \in A$ and $R \in \tau$ of arity $k$, if $(a_1,\dots,a_k) \in R^{\bA}$, then 
$R(f(a_1),\dots,f(a_k)) \in R^{\bB}$ 
We write $\bA \to \bB$ if there exists a homomorphism from $\bA$ to $\bB$, and we denote by 
$\CSP(\bB)$ the class of finite structures $\bA$ such that $\bA\to \bB$. 
Let ${\mathcal C}$ be a class of finite $\tau$-structures.
We say that ${\mathcal C}$ is 
\begin{itemize}
    \item \emph{closed under homomorphisms} if
    for every $\bB \in {\mathcal C}$, if $\bB \to \bA$, then $\bA \in {\mathcal C}$ as well.
    \item \emph{closed under inverse homomorphisms} if for every $\bB \in {\mathcal C}$, if $\bA \to \bB$, then $\bA \in {\mathcal C}$ as well (i.e., the complement of ${\mathcal C}$ in the class of all finite $\tau$-structures is closed under homomorphisms). 
\end{itemize}
Note that $\CSP(\bB)$ is closed under inverse homomorphisms. 

If $\bA$ and $\bB$ are $\tau$-structures with disjoint domains $A$ and $B$,
respectively, then the \emph{disjoint union} $\bA \uplus \bB$ is the $\tau$-structure $\bC$ with domain
$A \cup B$ and the relation $R^{\bC} = R^{\bA} \cup R^{\bB}$ for every $R \in \tau$. Note that
$\CSP(\bB)$ is \emph{closed under disjoint unions}, i.e., if $\bA \in {\mathcal C}$ and
$\bB \in {\mathcal C}$, then $\bA \uplus \bB \in {\mathcal C}$. 

\begin{observation}\label{obs:csp}
    Let $\tau$ be a finite relational signature. 
    A class of finite $\tau$-structures ${\mathcal C}$ is of the form $\CSP(\bB)$ for some
    countably infinite $\tau$-structure $\bB$ if and only if ${\mathcal C}$ is closed under
    inverse homomorphisms and disjoint unions.
\end{observation}

For examples we will often use graphs and digraphs. Since this paper is deeply involved
with logic and finite model theory, we will think of digraphs as binary structures with signature
$\{E\}$, following and adapting standard notions from graph theory~\cite{bondy2008} to this setting. 
In particular, given a digraph $\bD$, we call $E^\bD$ the \emph{edge set} of $\bD$, and its
elements we call \emph{edges} or \emph{arcs}.
Also, a \emph{graph} is a digraph whose
edge set is a  symmetric relation, and an \emph{oriented graph} is a digraph whose edge
set is an anti-symmetric relation. Given a positive integer $n$, we denote by 
$K_n$ the complete graph on $n$ vertices, i.e., the graph with vertex set $[n]$
and edge set $(i,j)$ where $i\neq j$. A \emph{tournament} is an oriented graph whose symmetric
closure is a complete graph,  and an \emph{oriented path} is an oriented graph whose symmetric
closure is a path. We denote by $\vec{P_n}$ the \emph{directed path} on $n$ vertices, 
 i.e., the oriented path with vertex set $[n]$ and edges $(i,i+1)$ for $i\in[n-1]$.

\subsection*{Fragments of first-order logic}
A first-order $\tau$-formula $\phi$ is called \emph{positive} if it does not use the negation symbol $\neg$
(so it just uses the logical symbols $\forall,\exists,\wedge,\vee,=$, variables and symbols from $\tau$). 
The negation of a positive formula is called \emph{negative}; note  that every negative formula is equivalent
to a formula in prenex conjunctive normal form where every atomic formula appears in negated form, and all
negation symbols are in front of atomic formulas; such formulas will also be called negative.

Consider a quantifier-free formula $\psi$ using only variables, the symbols $\land$ and $=$, and
symbols from $\tau$, and let $\psi'$ be the formula obtained from $\psi$  by  iteratively removing each
equality conjunct $x = y$ and substituting each occurrence of the variable $y$ by the variable $x$.
We now consider the structure $\bC_\psi$ whose vertex set is the set of
variables appearing in $\psi'$, and the interpretation of $R\in \tau$ consists of those tuples
$\overline x$ such that  $\psi'$ contains the positive literal $R(\overline x)$.
We say that a negative first-order sentence $\phi$ in prenex conjunctive normal form
is  \emph{connected} if for every clause $\varphi$  of the quantifier-free part of $\phi$, the structure
$\bC_{\lnot \varphi}$ is connected.

\section{Hereditary model checking and CSPs}
\label{sec:CSP}
In this section we see that hereditary model checking naturally arises in the context
of constraint satisfaction problems. We do so by providing several examples of $\CSP$s
in $\HerFO$. We also include inexpressibly examples of $\HerFO$ which help build intuition
about hereditary model checking. 

It is well-known that a universal first-order sentence $\phi$ in conjunctive normal form
describes a $\CSP$ if $\phi$ is negative and connected.
The following observation claims that $\HER(\phi)$  describes a $\CSP$ 
whenever $\phi$ is negative and connected. Note that we do not require that $\phi$ is  universal;  
in the case of universal sentences $\phi$, we have that $\HER(\phi)$ equals the class of finite models of $\phi$.

\begin{observation}\label{obs:HERFO-CSP}
    The following statements hold for every first-order formula $\phi$ in prenex conjunctive normal form.
    \begin{itemize}
        \item If $\phi$ is negative, then $\HER(\phi)$ is closed under inverse homomorphisms. 
        \item If $\phi$ is negative and connected, then there is a structure $\bS$ such that
        $\HER(\phi) = \CSP(\bS)$. 
    \end{itemize}
\end{observation}
\begin{proof}
    To prove the first itemized statement, suppose that there is a homomorphism $f\colon \bA\to \bB$
    and $\bB\in \HER(\phi)$. Let $\bA'$ be a substructure of $\bA$. The substructure of $\bB$ with
    vertex set $f[A']$ models $\phi$. Since surjective homomorphisms preserve positive first-order
    formulas,  we conclude that $\bA'$ models $\phi$. 
    
    It is straightforward to observe that if $\phi$ is negative and connected, then the class of
    finite models of $\phi$ is closed under disjoint unions. Hence, if every substructure of $\bA$ and every substructure
    of $\bB$ satisfy $\phi$, then every substructure of $\bA \uplus \bB$ satisfies $\phi$, i.e.,
    $\HER(\phi)$ is closed under disjoint unions.
    By the first claim, $\HER(\phi)$ is also closed under inverse homomorphisms, and it thus follows
    that there is a structure $\bB$ such that $\CSP(\bB) = \HER(\phi)$ ---
    for any class ${\mathcal C}$ of $\tau$-structures there exists a $\tau$-structure $\bS$ such that
    ${\mathcal C} = \CSP(\bS)$ if and only if ${\mathcal C}$ is closed under disjoint unions and inverse
    homomorphisms (see, e.g.,~\cite{Book}).
\end{proof}

\begin{remark}\label{rem:MSO}
    Every $\CSP$ in $\HerFO$  is the $\CSP$ of an $\omega$-categorical structure. This follows from a result in~\cite{BKR},
    which states that every $\CSP$ in monadic second-order logic (MSO) is the $\CSP$ for an $\omega$-categorical structure;
    clearly,  $\HerFO$ is a fragment of monadic second-order logic (see, e.g., Observation~\ref{obs:HerFO-UMSO}).
\end{remark}

\begin{example}\label{ex:CSP-N=neq}
    Let $\tau = \{N, EQ\}$ be the signature where $N$ and $EQ$ are two binary relational
    symbols which will encode ``equal'' and ``not equal'', respectively. 
    Note that a $\tau$-structure $\bA$ belongs to $\CSP(\mathbb N, =, \neq)$
    if and only if it contains no loop $N(x,x)$ and no \emph{contradicting cycle}, i.e., vertices $x_1,\dots, x_n$
    such that $EQ(x_1,x_{i+1})$ for all $i\in [n-1]$ and $N(x_1,x_n)$. Despite the fact that the existence of such a cycle is not a first-order property, 
    the problem $\CSP(\mathbb N,=, \neq)$ is in
    $\HerFO$.
    For a positive integer $k$ we write $d_{EQ}(x) = k$ to denote the first-order
    formula stating ``$x$ has exactly $k$ neighbours in the relation $EQ$ (different from $x$)'',
    and $d_{EQ}(x) \neq k$ for its negation. Now, consider the formula 
    \[
    \phi:= \forall x,y \exists z(\lnot N(x,y) \lor d_{EQ}(x)\neq 1 \lor d_{EQ}(y)\neq 1 \lor d_{EQ}(z) \neq 2).
    \]
    If $\bA$ does not hereditarily model $\phi$, then there is a subset $\{a_1,\dots, a_n\}$
    such that $N(a_1, a_n)$ and every $a_i$ has degree exactly $2$ in the relation $EQ$ except
    for $a_1$ and $a_n$, and so $\bA$ contains a contradicting cycle. Conversely, if $\bA$ has a
    contradicting cycle one can find a substructure
    $\bB$ of $\bA$ that does not model $\phi$, namely, any shortest contradicting cycle. 
    Therefore, $\CSP(\mathbb N, =, \neq)$ is hereditarily defined by the first-order sentence
    $\phi \land \forall x. \lnot N(x,x)$.
\end{example}

Some further examples of polynomial-time solvable CSPs in HerFO include the following
(see Appendix~\ref{ap:P-examples} for details).
\begin{itemize}[itemsep = 0.8pt]
    \item The class of digraphs that represent the cover relation of a poset (Example~\ref{ex:cover-relation}):
    this and $\CSP(\mathbb Q, <)$ (Example~\ref{ex:DAG}) are examples of infinite-domain CSPs in HerFO but not in FO. 
    \item $\CSP(\vec{P_3})$ (Example~\ref{ex:P2}):  a finite-domain CSP in HerFO but not in FO.
    \item $\CSP(\mathbb Q,\{(x,y,z)\colon x <\max\{y,z\})$ (Example~\ref{ex:and-or-scheduling}):
    an infinite-domain CSP in HerFO but not even in Datalog.
\end{itemize}

\subsection*{Hard examples}
In this paper we include two examples of coNP-complete CSPs in HerFO. Here, we 
present the example that we will also use in Corollaries~\ref{cor:undecidablity-fragments}
and~\ref{cor:Q-prefix-CSP}; to 
see the other one (possibly the better-known one as well) see Appendix~\ref{ap:Henson}.

Let $\tau$ consist of two binary symbols $E_b$ and $E_r$; we think of a $\tau$-structure
is a digraph with blue and red edges. For a positive integer
$n\ge 2$ we denote by $\mathbb{TD}_n$ the structure with vertex set $[n]$ such that $([n],E_b)$ is
a Hamiltonian directed cycle $1,\dots, n$, and $([n],E_r)$ is a complete (symmetric) graph.
Let $\mathcal T$ be the set containing the one-element structure with a red loop, and the one-element structure with a blue loop, and
all structures $\mathbb{TD}_n$ for $n\ge 2$. By Observation~\ref{obs:csp}, the class  $\Forb(\mathcal T)$, i.e., the
class of loopless $\tau$-structures $\bA$ for which there is no homomorphism
$\mathbb{TD}_n\to \bA$ for any $n\ge 2$, is a CSP, i.e., is of the form $\CSP(\bB_\mathcal T)$ for some $\tau$-structure $\bB_\mathcal T$.
The reader familiar with Fra\"isse limits may 
notice that $\bB_{\mathcal T}$ can be chosen to be a countable homogeneous structure.  We show that $\CSP(\bB_{\mathcal T})$
is a coNP-complete CSP in $\HerFO$. Consider the $\tau$-sentence
\[
\phi_\mathcal T:= \exists x,y\forall a \;(\lnot E_b(a,a)\land \lnot E_r(a,a) \land (\lnot E_b(x,a) \lor (x\neq y \land \lnot E_r(x,y)))).
\]
Observe that for every positive integer $n\ge 2$ the
structure $\mathbb{TD}_n$ does not satisfy $\phi_\mathcal T$,
and  neither do loops.
Also,  since
$\phi_\mathcal T$ is negative and connected,
it follows by Observation~\ref{obs:HERFO-CSP}
that if $\mathbb F\to\bA$ for some structure $\bF\in\mathcal T$, then $\bA$ does not hereditarily
satisfy $\phi_\mathcal T$. On the other hand, observe that if
a $\tau$-structure $\bA$ does not hereditarily satisfy $\phi_\mathcal T$, then $\bA$
contains a loop, or there is a subset $A'\subseteq A$ such that
$(A',E_r^{\bA'})$ is a complete symmetric graph with at least two vertices, and 
every vertex $x \in A'$ has a blue out-neighbour.
So if $\bA$ contains no loops, then
the shortest directed blue cycle in $A'$ induces a structure isomorphic to
$\mathbb{TD}_n$ for some $n\ge 2$. Therefore, $\HER(\phi_\mathcal T) = \Forb(\mathcal T) = \CSP(\bB_\mathcal T)$.

\begin{theorem}\label{thm:Forb-TD}
    $\Forb(\mathcal T)$ is a $\coNP$-complete $\CSP$ hereditarily definable by
    an $\exists\exists\forall$- and by an $\exists\forall\exists$-sentence.
\end{theorem}
\begin{proof}
   First notice that if $\phi_\mathcal T'$ is the sentence obtained
   from $\phi_\mathcal T$ by changing the prefix $\exists x,y\forall a$
   to $\exists x\forall a \exists y$, then $\phi_\mathcal T$ and $\phi_\mathcal T'$ are
   equivalent sentences. Hence, it follows from the discussion above
   that $\Forb(\mathcal T)$ is hereditarily definable by an $\exists\exists\forall$-
   and by an $\exists\forall\exists$-sentence. We now show that $\Forb(\mathcal T)$
   is $\coNP$-complete. 
   Consider an instance $\psi$ of 3SAT with variables $V$ and clauses
    $C_1,\dots, C_m$, where $C_i = (c_i^1, c_i^2, c_i^3)$ and $c_i^k \in \{v,\lnot v\}$
    for some $v\in V$. We construct a $\tau$-structure $\bA$
    with vertices $a_i^j$ for each $i\in[m]$ and $j\in [3]$.
    The blue edges of $\bA$ consist of all pairs $(a_i^j, a_{i+1}^k)$ and $(a_m^j,a_1^k)$,
    for $i\in [m-1]$ and $j,k\in [3]$; the red edges of $\bA$ correspond to the
    relation $c_i^k\neq \lnot c_j^l$, i.e., $(a_i^k, a_j^l)\in E_r^\bA$ if and only
    if the literal $c_i^k$ does not equal the negation of the literal $\lnot c_j^l$
    --- in particular, $E_r^\bA$ is a symmetric relation.
    Clearly, $\bA$ is a loopless $\tau$-structure. 
    
    We claim that 
    $\psi$ is satisfiable if and only if there is a homomorphism
    $\mathbb{TD}_n\to \bA$ for some $n\ge 2$. Suppose there is a
    satisfying assignment for $\psi$, and consider the vertices
    $a_i^{k_i}$ where $c_i^{k_i}$ is true in the clause $C_i$. 
    Then  
    the substructure $\bA'$ with domain $a_1^{k_1},\dots, a_m^{k_m}$ satisfies
    that every vertex has a blue out-neighbour. Clearly, $c_i^{k_i}$
    cannot be the negation of $c_j^{k_j}$, so $(A', E_r^{\bA'})$ is a complete
    red graph. This shows that if $\psi$ is satisfiable, then $\bA\not\in \Forb(\mathcal T)$. 
    
    Conversely, notice that if there is a substructure $\bA'$ of $\bA$ that satisfies that every vertex
    has a blue out-neighbour, then $A'$ contains a vertex $a_i^{k_i}$ for each $i\in[m]$.
    Moreover, if $(A',E_r^{\bA'})$ is a complete graph, then $A'$ contains at most
    one vertex $a_i^{k_i}$ for each $i\in[m]$. With similar arguments as before, 
    one can notice the $\psi$ is satisfiable by considering the evaluation  $f\colon V\to \{0,1\}$
    defined by $f(v) = 1$ if there is some clause $C_i$ such that
    $v = c_i^k$ and $a_i^k\in A'$.
\end{proof}

\subsection*{Inexpressible examples}
In this section we study the limitations of the expressive power of HerFO. Clearly, a class $\mathcal C$ is first-order definable if and only if 
the complement of $\calC$ is first-order definable.
A structure $\bA$ is a \emph{minimal obstruction} of a hereditary class $\calC$ if $\bA\not\in \calC$ but every
proper substructure $\bA'$ of $\bA$ belongs to $\calC$. 
We show 
that $\calC \in \HerFO$ if and only if the complement of $\calC$ contains a first-order
definable subclass $\calF'$ that contains all minimal obstructions of $\calC$ (Lemma~\ref{lem:trivial-HerFO}). 

We then apply this observation to show that the class of bipartite graphs
is not in HerFO (Example~\ref{ex:CSP-K2}).It is well-known that a graph is
bipartite if and only if it does not contain an odd cycle, and that the class
of odd cycles cannot be expressed by a first-order formula.
However, this is not enough to show that the class of  bipartite graphs is not in
HerFO: there are properties $\mathcal F$ that are not first-order definable, but
the class of all ${\mathcal F}$-free structures is in $\HerFO$; see Example~\ref{ex:CSP-N=neq}. 
To prove that the class of bipartite graphs is not in HerFO, we therefore need the following lemma.

\begin{lemma}\label{lem:trivial-HerFO}
    Let $\calC$ be a hereditary class of finite $\tau$-structures  and let $\calF$ be the class 
    of minimal obstructions of $\calC$. Then $\calC$ is hereditarily first-order
    definable if and only if there is a first-order sentence $\psi$ such that
    \begin{itemize}
        \item if $\mathbb F\in \calF$, then $\mathbb F\models \psi$, and
        \item if $\bA$ is a finite $\tau$-structure such that $\bA\models \psi$, then there is an embedding $\mathbb F\hookrightarrow \bA$
        for some $\mathbb F\in \calF$.
    \end{itemize}
\end{lemma}
\begin{proof}
    Suppose that $\calC = \HER(\phi)$ for some first-order formula $\phi$. We claim that $\psi := \neg \phi$ satisfies
    both itemized statements. Firstly, note that if $\bF \in {\mathcal F}$,  then $\bF \notin \HER(\phi)$, so 
    some substructure of $\bF$ does not satisfy $\phi$.
    Moreover, all substructures of $\bF$ belong to $\HER(\phi)$, so we have that $\bF$ itself does not satisfy $\phi$, and hence satisfies $\psi$. 
    If $\bA$ is a finite $\tau$-structure such that $\bA \models \psi$, then $\bA \notin {\mathcal C}$, and hence there exists $\bF \in {\mathcal F}$
    which embeds into $\bA$. This shows the forward implication of the statement. Conversely, if there exists a formula $\psi$ that satisfies both items of the statement, then 
    it is similarly straightforward to show that 
    ${\mathcal C} = \HER(\neg \psi)$.
\end{proof}

Building on this simple lemma we can now use standard Ehrenfeucht-Fra{\"i}ss\'e arguments
to show that certain hereditary classes are not in HerFO.
If $\bA$ and $\bA'$ are $\tau$-structures, we write $\bA \equiv_k \bA'$ if $\bA$ and $\bA'$ satisfy the same
first-order $\tau$-sentences with at most $k$ variables.

\begin{example}[Bipartite graphs not in $\HerFO$]\label{ex:CSP-K2}
    It is well known that the minimal obstructions of the class of bipartite graphs 
    are all odd symmetric cycles and the non-symmetric edge. For every positive integer $k$,
    there is a large enough odd cycle $\bC$ and a large enough even cycle $\bC'$ such that
    $\bC \equiv_k \bC'$ (this can be shown by an Ehrenfeucht-Fra{\"i}ss\'e argument; see, e.g.,~\cite{EbbinghausFlum}).
    Hence, we conclude via Lemma~\ref{lem:trivial-HerFO} that the class of bipartite graphs is
    not in $\HerFO$. Moreover, note that $\CSP(K_2)$ is the class of bipartite  digraphs.
    So it also follows from the existence of such cycles $\bC\equiv_k\bC'$ and Lemma~\ref{lem:trivial-HerFO}
    that $\CSP(K_2)$ is not in $\HerFO$.
\end{example}

In the appendix we use Lemma~\ref{lem:trivial-HerFO} to show that $\CSP(\mathbb Q, <, =)$ is not in $\HerFO$ (Example~\ref{ex:Q<=}).

\section{Undecidability of the tractability problem}
\label{sect:undec}
The \emph{tractability problem} for $\HerFO$ asks whether $\HER(\phi)$ can be solved in polynomial time
for a given first-order sentence $\phi$. In this section we show that if $\P\neq\NP$, then the
tractability problem for $\HerFO$ is undecidable. We begin with the following simple observation.

\begin{observation}\label{obs:FO-relative}
Consider a first-order $\tau$-sentence $\phi:=Q_1x_1\dots Q_nx_n. \psi(x_1,\dots, x_n)$. 
If $\xi(x)$ is a first-order $\tau$-formula, then there is a $\tau$-sentence $\phi_\xi$
such that a $\tau$-structure $\bA$ satisfies $\phi_\xi$ if and only if there is no element
$a\in A$ that satisfies $\xi$ or the substructure of $\bA$ with domain $\{a\in  A\colon \bA\models \xi(a)\}$ 
satisfies $\phi$. Namely, if $Q_1 =\dots =Q_n = \exists$, then 
 $$\phi_\xi:=\forall y.\lnot \xi(y) \lor \exists x_1,\dots, x_n (\bigwedge_{i\in [n]} \xi(x_i)\land \psi(x_1,\dots, x_n)).$$
 Otherwise, $$\phi_{\xi} := Q_1x_1\dots Q_nx_n \left ( \bigwedge_{i\in U}\xi(x_i) \implies \bigwedge_{i\in E}\xi(x_i)\land \psi(x_1,\dots, x_n) \right)$$ 
where $U$ (respectively, $E$) is the set of indices $i\in \{1,\dots,n\}$ such that $Q_i$ is a universal
(respectively, existential) quantifier.
\end{observation}

Similarly, if $U$ is a monadic predicate and $\phi$ is a first-order sentence, 
then $\phi_U := \phi_{U(x)}$ denotes the relativization of $\phi$ to the vertices in the set $U$, and
$\phi_{\lnot U} := \phi_{\neg U(x)}$ the relativization of $\phi$ to the complement of $U$. 

\begin{theorem}\label{thm:undecidability}
    If $\P\neq \coNP$, then it is undecidable to test whether the hereditary
    model-checking problem for a given first-order sentence $\phi$ is solvable in polynomial time. 
\end{theorem}
\begin{proof}
    We use Trakhtenbrot's theorem, which states that there is no algorithm that decides whether a given first-order formula
    $\phi$ has a finite model~\cite{Trakhtenbrot}. We reduce this decision problem to our problem.
    
    Let $\phi$ be a first-order $\tau$-sentence, let
    $U$ be a monadic predicate not in $\tau$, and let $E$ a binary predicate also not in $\tau$.
    Let $\psi$ be a first-order $\{E\}$-sentence such that the hereditary model-checking
    problem for $\psi$ in $\coNP$-complete, e.g., $\psi$ can be chosen to be the sentence hereditarily describing
    the CSP from Theorem~\ref{thm:Forb-TD}. Building on the relativizations  $\psi_{\lnot U}$
    and $(\lnot\phi)_U$ we define the following first-order $\tau\cup\{U,E\}$-sentence
    \[
        \chi:=  (\lnot \phi)_U \lor \psi_{\lnot U}.
    \]
    We claim that if $\phi$ does not have a finite model, then $\HER(\chi)$ is
    polynomial-time solvable, and if $\phi$ has a finite model, then $\HER(\chi)$
    is $\coNP$-complete. We first observe that if $\phi$ does not have a finite model,
    then $\chi$ is valid on all finite $\tau\cup\{U,E\}$-structures. 
    To see this, let $\bA$
    be a $\tau\cup\{U,E\}$-structure. If $U^\bA = \varnothing$,
    then $\bA$ satisfies
    the first disjunct $(\lnot\phi)_U$ of $\chi$ (Observation~\ref{obs:FO-relative}),
    and if $U^\bA\neq \varnothing$, then the
    substructure induced by $U^\bA$ also satisfies $(\lnot\phi)_U$, because
    $\phi$ has no finite models. In particular, this implies that $\HER(\chi)$ 
    can be solved in polynomial-time, because every instance is a yes-instance.

    Now suppose that $\phi$ has a finite model. Let $\bS$ be the smallest
    model of $\phi$, i.e., every proper substructure of $\bS$ satisfies $\lnot\phi$.
    We present a polynomial-time reduction from $\HER(\psi)$ to $\HER(\chi)$, which implies  that $\HER(\chi)$ is
    coNP-complete. Given an $\{E\}$-structure $\bA$ we consider the $\tau\cup\{U,E\}$-structure 
    $\bB$ defined as follows:
    \begin{itemize}
        \item the domain $B$ of $\bB$ is the disjoint union $A\cup S$, 
        \item the interpretation of $U$ in $\bB$ is $S$, 
        \item the interpretation of $E$ in $\bB$ equals $E^\bA$, and
        \item for every $R\in \tau$, the interpretation of $R$ in $\bB$ is $R^\bS$.
    \end{itemize} 
    The structure $\bB$ can be computed from the structure $\bA$ in polynomial time (the structure  ${\bS}$ is constant).
    We now show that every substructure of $\bA$ satisfies $\psi$ if and only
    if every substructure of $\bB$ satisfies $\chi$. First, note that
    if a substructure $\bC$ of $\bB$ does not contain $\bS$, then $\bC$ satisfies
    $\chi$. Indeed, if $C\subseteq A$, then there is no element of $\bC$ that models
    $U(x)$, so by Observation~\ref{obs:FO-relative} it follows that $\bC\models (\lnot \phi)_U$
    and hence $\bC \models \chi$. Otherwise,  if $\varnothing \neq C\cap S \neq S$, then the substructure
    $\bC'$ of $\bC$ with domain $U^\bC$ satisfies $\lnot \phi$,  because $\bC'$ is a proper
    substructure of $\bS$, and $\bS$ is the smallest model of $\phi$. Hence, $\bC\models (\lnot \phi)_U$,
    and so $\bC\models \chi$. These observations imply that 
    every substructure $\bC$ of $\bB$ satisfies $\chi$ if and only if every
    substructure $\bD$ of $\bB$ with $S\subseteq D$ satisfies $\chi$.
    If $S = D$, then $U^\bD = D$, so no element $d$ of $\bD$ satisfies $\lnot U(d)$, 
    and similarly as above, we conclude that $\bD\models \chi$ because $\bD\models \psi_{\lnot U}$.
    Finally, we assume that $S\subseteq D$ and $D\neq S$, and  so $D\cap A\neq\varnothing$. Since
    $S\subseteq D$ and $\bS\models \phi$, we have that $\bD$ does not satisfy $(\lnot\phi)_U$.
    Hence, $\bD\models\chi$ if and only if $\bD\models \psi_{\lnot U}$, 
    and the latter holds if and only if the substructure of $\bA$ with domain 
    $D\cap A$ satisfies $\psi$. 
    We thus conclude that $\bB\in \HER(\chi)$ if and only if $\bA\in \HER(\psi)$. 
    
    All together this shows that if
    $\P\neq \coNP$, then $\phi$ has a finite model if and only if $\HER(\chi)$
    is not solvable in polynomial time.
\end{proof}

A \emph{quantifier prefix} is a word $Q\in \{\exists,\forall\}^\ast$, and we say
that a first-order formula $\phi$ in prenex normal form has quantifier prefix $Q$
if $\phi = Q_1x_1\dots Q_nx_n.\psi$ where $Q = Q_1\dots Q_n$ and $\psi$ is a quantifier-free
formula. We say that two quantifier prefixes are \emph{dual} to each other if one can be obtained from the other by
exchanging the symbols $\forall$ and $\exists$.

\begin{corollary}\label{cor:undecidablity-fragments}
    Consider a quantifier prefix 
    $Q\in \{\exists,\forall\}^\ast$, 
    and assume $\P\neq\NP$. If the finite satisfiability problem for first-order sentences whose quantifier prefix is dual to $Q$ is undecidable, 
    then the tractability problem for $\HerFO$ remains undecidable even for first-order sentences with quantifier prefix
    $Q$. In particular,  this is the case if $Q$ contains $\exists\forall\exists$ or $\exists^3\forall$ as subwords.
\end{corollary}
\begin{proof}
    Denote by $Q'$ the dual of $Q$. It is known that if $Q'$ is of the form $\exists^\ast\forall\exists^\ast$ or of the form
    $\forall^\ast \exists^\ast$, then the satisfiability problem is decidable for the fragment of FO with quantifier prefix $Q'$~\cite{Ackermann,BernaysSchoenfinkel}.
    Hence, if the finite satisfiability problem is undecidable for $Q'$, then $Q'$ contains $\exists\forall \exists$
    or $\forall\forall\exists$ as a subword. By Theorem~\ref{thm:Forb-TD}, we can choose an FO sentence $\psi$
    with quantifier prefix $Q$ such that $\HER(\psi)$ is $\coNP$-complete. Now,
    notice that in the proof of Theorem~\ref{thm:undecidability} one can choose
    $\chi:= (\lnot\phi)_U\lor \psi_{\lnot U}$
    to have quantifier prefix of the form $Q$:
     \begin{itemize}
         \item since the quantifier prefix of $\phi$ is dual to $Q$, the quantifier prefix of $\neg \phi$ is $Q$, 
         \item by assumption, $\psi$ has quantifier prefix $Q$ and $\HER(\psi)$ is coNP-complete,
         \item since $Q'$ contains an existential quantifier, $Q$ contains a universal quantifier, and so
         $(\lnot\phi)_U$ and  $\psi_{\lnot U}$ have the same quantifier prefixes as $\phi$ and $\psi$, respectively
         (Observation~\ref{obs:FO-relative}),
         \item finally, by moving quantifiers to the front, we can rewrite $\chi$ with quantifier prefix
         $Q$.
     \end{itemize}
     In this way, following the proof of Theorem~\ref{thm:undecidability} we conclude that for a given $\phi$
     the problem $\HER(\chi)$ is polynomial-time solvable if and only if $\phi$ has no finite models, and hence undecidable by assumption.
     The last statement of this corollary follows because the finite satisfiability problem for 
     the fragments of FO with quantifier prefix $\exists\forall\exists$ or $\exists^3\forall$ is undecidable
     (see, e.g.,~\cite[Theorem 3.0.1]{DecisionProblem}. 
\end{proof}

\section{Hereditary model checking and quantifier prefixes}
\label{sect:prefixes}

In this section we present the following dichotomy for quantifier prefixes: 
for every quantifier prefix $Q\in\{\exists,\forall\}^\ast$ either 
\begin{itemize}
    \item $\HER(\phi)$ is in $\PO$ for every first-order formula $\phi$ with quantifier prefix $Q$, or
    \item there exists a first-order formula $\phi$ with quantifier prefix $Q$ 
    such that $\HER(\phi)$  is coNP-complete.
\end{itemize}
Moreover, we show that in the former case, $\HER(\phi)$ is also contained in SNP. 

A relational signature $\tau$ is called \emph{monadic} if all relation symbols in $\tau$ are monadic.
It is easy to see that every problem in HerFO with a monadic signature is polynomial-time solvable.

\begin{proposition}\label{prop:HerFO-monadic}
    Let $\tau$ be a finite monadic relational signature. For every first-order formula 
    $\phi$ the class $\HER(\phi)$ is universally definable and hence in $\PO$. 
\end{proposition}
\begin{proof}
    We claim that $\bA$ hereditarily models $\phi$ if and only if every substructure 
    $\bB$ of $\bA$ with at most $2^{|\tau|}$ many elements satisfies $\phi$. One direction follows from
    definition of hereditary satisfiability. For the converse implication, let $\bB$ be a
    substructure of $\bA$ and consider a minimal subset $C\subseteq B$ such that
    for every $b\in B$ there is a $c\in C$ such that $U(b) \Leftrightarrow U(c)$
    for all $U\in \tau$. Clearly, $|C|\le 2^{|\tau|}$ and it is straightforward to observe
    that $\bB\models \phi$ if and only if $\bC\models \phi$, and the claim follows.
\end{proof}
 
From now on, we only consider the non-monadic case. The key components in the
proof of our classification (Theorem~\ref{thm:Q-classification}) are
Algorithm~\ref{alg:main2} (for one of the tractable cases) and the fact that 
the problem of deciding whether every directed cycle in an input digraph $\bD$ induces a symmetric
edge is coNP-complete (Theorem~\ref{thm:hard-cycles}) and expressible in HerFO (Lemma~\ref{lem:EEU}).

\subsection*{The $\mathbf{\forall^\ast\exists^\ast}$ fragment}

In this subsection we prove that for every $\forall^\ast\exists^\ast$-formula
$\phi$ there is a universal formula $\phi'$ such that a structure $\bA$
hereditarily satisfies $\phi$ if and only if $\bA\models \phi'$. 

\begin{lemma}\label{lem:collapse}
    Let $\phi$ be a $\forall^\ast\exists^\ast$-formula with $k$ universally quantified variables. 
    Then a structure $\bA$
    hereditarily models $\phi$ if and only if every $k$-element substructure of $\bA$
    models $\phi$. 
\end{lemma}
\begin{proof}
    We prove the non-trivial (but straightforward) implication. Suppose that every
    substructure $\bB$ of $\bA$ with $|B|\le k$ models $\phi$, and let $\bA'$
    be a substructure of $\bA$. If $|A'|\leq k$, then $\bA'\models \phi$; otherwise,
    for a $k$-tuple $(a_1,\dots, a_k) \in (A')^k$, let $\overline b$ be a
     tuple such that the quantifier-free part $\psi$ of $\phi$ is true
    of $(a_1,\dots, a_k,\bar{b})$ in the substructure of $\bA'$ with  vertex set 
    $\{a_1,\dots, a_k\}$. It follows that $\bA'\models \psi(a_1,\dots, a_k,\bar{b})$,
    and since such a $\bar{b}$ exists for every $\bar{a}\in (A')^k$, we conclude
    that $\bA'\models \phi$, and therefore $\bA\in \HER(\phi)$.
\end{proof}

\begin{corollary}\label{cor:U*E*}
    If $\phi$ is a $\forall^\ast \exists^\ast$-sentence, then $\HER(\phi)$ 
    is universally definable and hence polynomial-time solvable. 
\end{corollary}
\begin{proof}
    If $\phi = \forall x_1,\dots,x_k \exists y_1,\dots,y_l. \, \psi$ where $\psi$ is quantifier-free, 
    let $\phi'$ be the formula
    $$\forall x_1,\dots,x_k \exists y_1 \in \{x_1,\dots,x_k\}, \dots, y_l \in \{x_1,\dots,x_k\} . \, \psi.$$ 
    By Lemma~\ref{lem:collapse}, a structure hereditarily satisfies
    $\phi$ if and only if it satisfies $\phi'$. 
\end{proof}

\subsection*{The $\mathbf{\forall^\ast\exists\forall^\ast}$ fragment}
An \emph{SNP $\tau$-sentence} (short for \emph{strict non-deterministic polynomial-time})
is a sentence of the form 
$$ \exists R_1,\dots,R_k \forall x_1,\dots,x_n. \psi$$
where $\psi$ is a quantifier free $\tau \cup \{R_1,\dots,R_k\}$-formula. 
If a structure $\bA$ satisfies the sentence $\Psi$, we write $\bA \models \Psi$.
We say that a class of finite $\tau$-structures $\mathcal C$ is in \emph{SNP} if
there exists an SNP $\tau$-sentence $\Phi$ such that  $\bA \models \Phi$ if and only
if $\bA \in {\mathcal C}$.  We show that if $\phi$ is a $\forall^\ast\exists\forall^\ast$-formula, 
then $\HER(\phi)$ is in SNP $\cap$ P.

For this subsection we consider a fixed $\forall^\ast\exists\forall^\ast$-formula
\[
\phi=\forall x_1,\dots, x_k \exists y \forall x_{k+1},\dots, x_n. ~\psi(x_1,\dots, x_k,y,x_{k+1},\dots, x_n)
\]
where $\psi$ is a quantifier-free $\tau$-formula. We expand $\tau$ with an $(l+2)$-ary relation symbol
$L$. We will interpret $L$ as a reflexive linear order with $l$ parameters. It is straightforward
to observe that there is a universal $\{L\}$-formula $\Lin(x_1,\dots, x_l)$  such that
$\Lin$ is true of an $l$-tuple $\overline a$ in an $\{L\}$-structure $\bA$ if and only if
the binary relation $L({\overline a}, x,y)$ defines a reflexive linear order $x\le_{\overline a} y$
on $A$.
Consider now the SNP sentence defined as follows.
\begin{align}
\Phi:= \exists L \; \forall x_1,\dots, x_k, & y, x_{k+1},\dots, x_n. ~\Lin(x_1,\dots, x_k) \nonumber\\
   \wedge & \left (\bigwedge_{i\in [n]} L(x_1,\dots, x_k,y,x_i)\right) \Rightarrow \psi(x_1,\dots, x_k,y,x_{k+1},\dots, x_n).\nonumber
\end{align}

\begin{lemma}\label{lem:U*EU*}
    A finite $\tau$-structure $\bA$ hereditarily models $\phi$ if and only if it  models $\Phi$.
\end{lemma}
\begin{proof}
    For the easy direction, suppose that $(\bA,L)$ models the first-order part of $\Phi$. For all
    $a_1,\dots, a_k \in A$, let $b \in A$ be the minimum with respect to the linear order
    $\le_{a_1,\dots, a_k}$, i.e., the element $b \in A$ such that $L(a_1,\dots, a_k,b,c)$ for all
    $c\in A$. In particular,  for all $a_{k+1},\dots, a_n \in A$ and $i\in[n]$ the atomic formula
    $L(a_1,\dots, a_k,b, a_i)$ holds in $\bA$, and thus $\bA\models \psi(a_1,\dots, a_k, b, a_{k+1},\dots, a_n)$.
    Since $\Phi$ is an SNP sentence, every substructure $\bB$ of $\bA$ also models
    $\Phi$, and by the previous argument we conclude that $\bB$ models $\phi$. Hence, 
    $\bA$ hereditarily models $\phi$. 
    
    Conversely, suppose that $\bA$ hereditarily satisfies $\phi$. For every $k$-tuple
    ${\overline{a}} = (a_1,\dots, a_k)$ of $A$ we define a reflexive linear order $\le_{\overline a}$
    such that the expansion $$(\bA,\{(a_1,\dots, a_k,b,c)\colon b \le_{(a_1,\dots, a_k)} c\})$$ 
    models the first-order part of $\Phi$. Let $b_1\in A$ be any element witnessing that 
    $\bA$ satisfies $$\exists y\forall x_{k+1},\dots, x_n.\psi(\overline a, y, x_{k+1},\dots, x_n).$$
    For $l > 1$, if no $b_i$ is a coordinate of  $\overline a$ for $i< l$, 
    choose $b_l$ to be any element witnessing that the substructure of $\bA$
    with domain $A\setminus\{b_1,\dots, b_{l-1}\}$ satisfies 
    $\exists y\forall x_{k+1},\dots, x_n.\psi(\overline a, y, x_{k+1},\dots, x_n)$ 
    (such a vertex $b_l$ exists since $\bA$ hereditarily satisfies $\phi$).  Otherwise, if some
    $b_i$ equals some coordinate
    of ${\overline a}$, then let $b_l$ be an arbitrary element of $A\setminus\{b_1,\dots, b_{l-1}\}$.
    We define the linear ordering $b_i \le_{\overline a} b_j$ if and only if $i\le j$, and let
    $L:=\{(\overline a, b,c)\in A^{k+2}\colon b \le_{\overline a} c\}$.
    It follows from the definition of $L$ and of $\Lin$ that $(\bA,L)\models
    \forall a_1,\dots, a_k. \Lin(a_1,\dots, a_k)$. 
    
    Suppose that $(\bA,L)$ satisfies $\bigwedge_{i\in [n]} L(a_1,\dots, a_k,b,a_i)$
    for some $a_1,\dots, a_k, b, a_{k+1},\dots, a_n\in A$. Let $b_1,\dots, b_m$ be the enumeration
    of $A$ corresponding to the linear ordering $\le_{(a_1,\dots, a_k)}$, and suppose that $b = b_l$.
    Since $(\bA,L)\models\bigwedge_{i\in [n]} L(a_1,\dots, a_k,b,a_i)$, i.e., $b\le_{(a_1,\dots, a_k)}
    a_i$ for every $i\in[n]$, it must be
    the case that every $a_i$ belongs to $A\setminus \{b_1,\dots, b_{l-1}\}$. It thus follows from the definition of $b_l$
    that the substructure $\bB$ of $\bA$ with vertex set $A\setminus \{b_1,\dots, b_{l-1}\}$ models 
    $\psi(a_1,\dots, a_k,b_l,a_{k+1},\dots, a_n)$, and thus $(\bA,L)\models \psi(a_1,\dots,
    a_k,b_l,a_{k+1},\dots, a_n)$.
    This shows that $\bA$ satisfies $\Phi$.
\end{proof}

The proof of Lemma~\ref{lem:U*EU*} suggests  a polynomial-time algorithm
that on input structure $\bA$ finds  parameterized linear orderings
proving that $\bA$ hereditarily satisfies $\phi$, or finds a substructure $\bB$ of $\bA$
that does not model $\phi$ (which certifies that $\bA$ does not hereditarily satisfy $\phi$).

\medskip 
\noindent\textbf{Certifying polynomial-time algorithm.}
Consider a fixed first-order sentence 
\[
\phi:= \forall x_1,\dots, x_k\exists y. \phi'(x_1,\dots, x_k, y)
\]
where $\phi'$ is a universal formula.

\RestyleAlgo{ruled}
\SetAlgoVlined{}
\begin{algorithm}
\DontPrintSemicolon{}
\SetKwInOut{Input}{input}
\Input{a finite $\tau$-structure $\bA$ with domain $A = \{a_1,\dots, a_m\}$.}
\caption{Cert-Her-$\forall^*\exists\forall^*$}

\ForEach{$\bar a \in A^k$}{
$S := \varnothing$ and $\le \; := \varnothing$. \\
\Repeat{$s$ equals some coordinate of ${\overline a}$}{
Let $\bA'$ be the substructure of $\bA$ with domain $A' := A \setminus S$. \\
\If{$\bA'$ satisfies $\forall y. \neg \phi'(\bar a,y)$}{Return $\bA'$}
 Let $s \in A'$ be such that $\bA'$ satisfies
    $\phi'({\overline a},s)$. \\
    Update  $S:= S \cup\{s\}$ and $\le \; := \; \le \cup \, \{(x,s)\mid x \in S\}$.
    }
    Let $\leq_{\bar a} \; := \; \leq~\cup~S\times (A\setminus S)~\cup~\{(a_i,a_j) \in (A\setminus S)^2\mid  i \le j\}.$   
}
Let $L := \{(\overline a, b,c)\in A^{k+2}\colon \overline a\in A^k, b \le_{\bar a} c\}$.\\
Return $(\bA, L)$
\label{alg:main2}
\end{algorithm}

\begin{lemma}\label{lem:algorithm-main}
    For every first-order formula $\phi:= \forall x_1,\dots, x_k\exists y.\phi'(x_1,\dots, x_k, y)$
    where $\phi'$ is a universal formula $\forall x_{i+1},\dots,x_n. \psi$ for some quantifier-free formula
    $\psi$, and for every $\tau$-structure $\bA$ with domain $A = \{a_1,\dots, a_m\}$ the following
    statements hold. 
    \begin{itemize}
        \item If Algorithm~\ref{alg:main2} returns a substructure $\bA'$ of $\bA$, then
        $\bA'\models \lnot \phi$, and $\bA\not\in \HER(\phi)$.
        \item If Algorithm~\ref{alg:main2} returns an expansion $(\bA,L)$ of $\bA$, 
        then $(\bA,L)$ satisfies the first-order part of $\Phi$, and $\bA\in \HER(\phi)$.
    \end{itemize}
\end{lemma}
\begin{proof}
    To prove the first claim, notice that (by finite induction) at the \textbf{if} statement in the \textbf{repeat-until} loop
    of the algorithm, the set $A'$ contains all entries of $\bar a$.
    Hence, if the algorithm returns $\bA'$, then $\bA'$ does not satisfy
    $\exists y. \phi'(\bar a,y)$ and hence $\bA$ does not hereditarily satisfy $\phi$.
    Now we argue that the second itemized statement holds. It is straightforward to 
    observe that $\le_{\bar{a}}$ is a reflexive linear order for every $\bar{a}\in A^k$. Hence, 
    it follows from the definition of $L$ that $(\bA,L)\models \forall x_1,\dots, x_k.\Lin(x_1,\dots, x_k)$. To see that $(\bA,L)$ models the second conjunct of $\Phi$, let $\bar b \in A^k$, $c\in A$, and $b_{k+1},\dots, b_n\in A$ (so $(\bar{b}, c, 
    b_{k+1},\dots, b_n)$ is an evaluation of the universally quantified variables of $\Phi$ in $A$). Further, suppose that 
    \[
    (\bA, L)\models \bigwedge_{i\in [n]} L(b_1,\dots, b_k,c,b_i);
    \]
    otherwise the second conjunct in the definition of $\Phi$ is vacuously true for the tuple $(b_1,\dots, b_k,c$, $b_{k+1},\dots, b_n)$.
    By the definition of $L$, this means that $c \le_{\bar{b}} b_i$ for every $i\in [n]$. Hence, 
    there is some iteration of the \textbf{repeat-until} loop such that $c\in S$. Let $S_i$
    be the set $S$ and the end of this iteration $i$ of the loop, and $S_0:= \varnothing$.
    It follows from the definition of $\le_{\bar{b}}$ and the assumption that
    $c\le_{\bar{b}} b_j$
    for each $j\in[n]$, that if $S_i = S_{i-1}\cup \{c\}$,
    then $(\{b_1,\dots, b_n\}\cap S_{i-1}) = \varnothing$.
    Since $c$ was added to $S$ in the
    $i$-th iteration, it must be the case that in the $i$-th iteration the \textbf{if} statement
    is not true
    for $s:= c$, i.e., $\bA'\models \forall x_1,\dots, x_k.\phi'(x_1,\dots, x_k,c)$.   
    Since $\{b_1,\dots, b_n\}\subseteq A\setminus S_{i-1}$, we conclude that in particular
    $\bA'\models \phi'({\bar{b},c})$ where the universally quantified variables from $\phi'$
    are interpreted as $(b_{k+1},\dots, b_n)$. This means that 
    \[
    (\bA,L)\models  
    \left (\bigwedge_{i\in [n]} L(b_1,\dots, b_k,c,b_i)\right) \Rightarrow \psi(b_1,\dots, b_k,c,b_{k+1},\dots, b_n).
    \]
    Since this is true for any choice of elements $b_1,\dots, b_n$ and $c$ in $A$, we conclude that
    $(\bA,L)$ satisfies the first-order part of $\Phi$, and by Lemma~\ref{lem:U*EU*}
    we conclude that $\bA\in \HER(\phi)$.
\end{proof}

Clearly, Algorithm~\ref{alg:main2} runs in polynomial time in the representation size of $\bA$.
Hence, the following statement is an immediate consequence of Lemma~\ref{lem:algorithm-main}.

\begin{theorem}\label{thm:U*EU*}
    If $\phi$ is a $\forall^\ast \exists \forall^\ast$-sentence, then there is an 
    $\SNP$ sentence $\Phi$ such that a finite structure $\bA$ satisfies $\Phi$ if and only
    if $\bA\in \HER(\phi)$. Moreover, $\Phi$ can be efficiently computed from $\phi$, and
    there is a polynomial-time algorithm that either finds an expansion of $\bA$ proving
    that $\bA\models \Phi$, or finds a substructure $\bA'$ of $\bA$ such that
    $\bA\models \lnot \phi$.
\end{theorem}

Note that this theorem covers the first-order sentences that show that Example~\ref{ex:forests},
Example~\ref{ex:chordal}, 
Example~\ref{ex:DAG}, and Example~\ref{ex:and-or-scheduling}  
are in $\HerFO$, because they are $\exists \forall^*$ sentences. 
Hence, the polynomial-time tractability of each of these problems follows from Theorem~\ref{thm:U*EU*}. 

A further natural example of a $\forall\forall\exists \forall^\ast$-sentence is given in Example~\ref{ex:cover-relation}
(the class of digraphs that corresponds to the cover relation of some poset).

\begin{figure}[ht!]
\centering
\begin{tikzpicture}

  \begin{scope}[yshift = 3cm, xshift = -1cm, scale = 0.8]
    \node [vertex] (l1) at (-1.5,1) {};
    \node [vertex] (l2) at (-2.25,0) {};
    \node [vertex] (l3) at (-0.75,0) {};
    \node [vertex] (l4) at (-1.5,-1) {};
    
    \node [vertex, label = right:$a$] (a) at (1.5,1) {};
    \node [vertex, label = right:$b$] (b) at (2.25,0) {};
    \node [vertex] (1) at (0.75,0) {};
    \node [vertex] (2) at (0.75,-1) {};
    \node [vertex] (3) at (2.25,-1) {};
    \node (L1) at (0,-1.5) {$\bG$};
      
    \foreach \from/\to in {l1/l2, l1/l3, l2/l3, l2/l4, l3/l4,
    a/b, a/1, 1/2, 2/3, 3/b, b/1}     
    \draw [edge] (\from) to (\to);
  \end{scope}

  \begin{scope}[xshift = -4.5cm, scale = 0.6]
    \node [vertex, label = below:$l_1$] (l1) at (-4,0) {};
    \node [vertex, label = below:$l_2$] (l2) at (-3,0) {};
    \node [vertex, label = below:$l_3$] (l3) at (-2,0) {};
    \node [vertex, label = below:$l_4$] (l4) at (-1,0) {};
    
    \node [vertex, label = below:$a$] (a) at (0,0) {};
    \node [vertex, label = below:$b$] (b) at (1,0) {};
    \node [vertex] (1) at (2,0) {};
    \node [vertex] (2) at (3,0) {};
    \node [vertex] (3) at (4,0) {};
    \node (L1) at (0,-2) {$(\bG,\le_a)$};

    \foreach \from/\to in {l1/l2,  l2/l3, l3/l4,
    a/b,  1/2, 2/3, b/1} 
    \draw [edge] (\from) to (\to);
     \foreach \from/\to in {l1/l3, l2/l4, a/1, b/3} 
    \draw [edge] (\from) to [bend left] (\to);
  \end{scope}

  \begin{scope}[xshift = 4.5cm, scale = 0.8]
  \node [vertex, label = below:$l_1$] (l1) at (-4,0) {};
    \node [vertex, label = below:$l_2$] (l2) at (-3,0) {};
    \node [vertex, label = below:$l_3$] (l3) at (-2,0) {};
    \node [vertex, label = below:$l_4$] (l4) at (-1,0) {};
    \node [vertex, fill = black, label = right:$a$] (a) at (1.5,1) {};
    \node [vertex, fill = black, label = right:$b$] (b) at (2.25,0) {};
    \node [vertex, fill = black] (1) at (0.75,0) {};
    \node [vertex, fill = black] (2) at (0.75,-1) {};
    \node [vertex, fill = black] (3) at (2.25,-1) {};
    \node (L1) at (0,-2) {$(\bG, <)$ and $\mathbb H$};
      
    \foreach \from/\to in {l1/l2,  l2/l3, l3/l4,
    a/b, a/1, 1/2, 2/3, 3/b, b/1} 
    \draw [edge] (\from) to (\to);
     \foreach \from/\to in {l1/l3, l2/l4} 
    \draw [edge] (\from) to [bend left] (\to);
  \end{scope}
  
\end{tikzpicture}
\caption{
Consider a first-order $\{E\}$-sentence $\phi$ that states that $E$ is a symmetric relation, and that 
for every vertex $x$ there is a vertex $y$ not adjacent to $x$ such that the neighbourhood of $y$ induces 
a clique. Clearly, $\phi$ can be chosen to be an $\forall \exists \forall^\ast$ formula. On the top, 
we depict a graph $\bG$ with two distinguished vertices $a$ and $b$. At the bottom left, we depict the linear
order $\le_a$ (from left to right) of $\bG$ obtained via Algorithm~\ref{alg:main2} (where the linear order between
the vertices greater that $a$ can be arbitrary), and that proves that $(\bG,a)$ hereditarily 
satisfies $\phi(a,y,\bar{z})$. At the bottom right, we depict the partial linear order $<$ and the subgraph $\bH$ (depicted
with black vertices) of $\bG$ that Algorithm~\ref{alg:main2} finds when running the loop for $x = b$, and hence,
proving that $(\bG,b)$ does not hereditarily satisfy $\phi$.}
\label{fig:U*EU*}
\end{figure}

\subsection*{$\exists\forall\exists$- and $\exists^2\forall$-sentences}
Theorem~\ref{thm:Forb-TD} already provides examples of first-order
sentences $\phi$ and $\psi$ with quantifier prefix $\exists\forall\exists$
and $\exists^2\forall$ such that $\HER(\phi)$ and $\HER(\psi)$ are coNP-complete. 
Moreover, $\phi$ and $\psi$ are formulas with a binary signature (with two relation symbol).
In this section we provide first-order sentences with the same quantifier prefixes over the signature of digraphs
(they only use \emph{one} binary symbol), which is the final ingredient needed
to prove Theorem~\ref{thm:Q-classification}.

Consider the class of digraphs $\bD$ such that for every directed cycle
$d_1,\dots, d_n$ of $\bD$ there exist $i,j\in [n]$ such that
$d_id_j$ is a symmetric edge of $\bD$, i.e., $(d_i,d_j),(d_j,d_i)\in \bD$.
In this case, we say that every directed cycle of $\bD$ `induces a symmetric
edge'.

\begin{lemma}\label{lem:EEU}
    Every directed cycle of a finite digraph $\bD$ induces a symmetric
    edge if and only if $\bD$ hereditarily satisfies the sentence
    \begin{align}
    \exists x,y\forall a \big (\lnot E(x,a) \lor [E(x,y)\land E(y,x)] \big ).
    \label{eq:symedge}
    \end{align}
\end{lemma}
\begin{proof}
    Suppose there is a directed cycle $d_1,\dots, d_n$ of $\bD$
    that does not induce a symmetric edge. Then the substructure
    of $\bD$ with vertex set $\{d_1,\dots, d_n\}$ satisfies
    the formula $$\forall x,y\exists a \big (E(x,a) \land [\lnot E(x,y)\lor \lnot E(y,x)] \big ),$$
    and so $\bD$ does not hereditarily satisfy~\eqref{eq:symedge}.
    
    Conversely, suppose that every directed cycle of $\bD$ induces a symmetric
    edge and let $B\subseteq D$. If $\bB$ contains a sink $b$, then
    $\bB\models \forall a\lnot E(b,a)$, and so $\bB$ satisfies~\eqref{eq:symedge}. 
    Otherwise, $\bB$ contains a directed cycle, and by assumption this directed
    cycle induces a symmetric edge $uv$, so by letting $x = u$ and $y = v$
    we conclude that $\bB$ satisfies $\exists x,y (E(x,y)\land E(y,x))$. Again, $\bB$ satisfies~\eqref{eq:symedge}.  
\end{proof}

Now, we prove that deciding whether every directed cycle of an input digraph $\bD$ 
induces a symmetric edge is coNP-complete. 

\begin{theorem}\label{thm:hard-cycles}
    The problem of deciding if every directed cycle in an input digraph $\bD$ induces
    a symmetric edge is $\coNP$-complete. 
\end{theorem}
\begin{proof}
    A hardness proof for this problem can be obtained with a script
    similar to the one used for the hardness proof of Theorem~\ref{thm:Forb-TD}. 
    We present the reduction, 
    and leave the proof of soundness and correctness to the reader.
    Consider an instance $\psi$ of 3SAT with variables $V$, clauses
    $C_1,\dots, C_m$ where each $C_i = (c_i^1, c_i^2, c_i^3)$ and $c_i^k \in \{v,\lnot v\}$
    for some $v\in V$. We construct a digraph $\bD$
    with vertices $s, t$ and $d_i^j$ for each $i\in[m]$ and $j\in [3]$.
    The edge set of $\bD$ consists of 
    \begin{itemize}
        \item all pairs $(s, d_1^k)$, $(d_i^j, d_{i+1}^k)$, $(d_m^j,t)$,
    and $(t,s)$ for $i\in [m-1]$ and $j,k\in [3]$, and 
    \item symmetric edges $(d_i^k, d_j^l)$ and $(d_j^l, d_i^k)$ for all 
    $i,j \in [m]$ and $k,l \in[3]$ such that $c_i^k = \lnot c_j^l$. 
    \end{itemize}
    Using similar arguments as in the proof of Theorem~\ref{thm:Forb-TD},
    one can verify that $\psi$ is satisfiable if and only if there is a directed
    cycle in $\bD$ that induces no symmetric edge. 
\end{proof}

\begin{corollary}\label{cor:EEU}
    There are $\exists^2\forall$ and $\exists \forall \exists$ (digraph) formulas $\phi$ and $\psi$
    such that  $\HER(\phi)$ and $\HER(\psi)$ are $\coNP$-complete.
\end{corollary}
\begin{proof}
    For $\phi$ consider the formula from Lemma~\ref{lem:EEU} and for $\psi$ consider
    the quantifier reordering $\exists x\forall a\exists y$ of $\phi$, and notice that
    $\phi$ and $\psi$ are logically equivalent. The claim now follows from 
    Lemma~\ref{lem:EEU} and Theorem~\ref{thm:hard-cycles}.
\end{proof}

\subsection*{Prefix classification}
The classification for general relational signatures follows from the lemmas proved earlier in this section. 

\begin{theorem}\label{thm:Q-classification}
    Let $\tau$ be a relational signature which is not monadic. For every quantifier prefix
    $Q\in \{\exists,\forall\}^\ast$ one of the following statements hold:
    \begin{itemize}
        \item $Q $ is of the form $\forall^\ast\exists^\ast$, or of the form $\forall^\ast\exists\forall^\ast$, and in this case
        $\HER(\phi)$ is in $\PO$ for every first-order $\tau$-sentence $\phi$ with quantifier prefix $Q$,  or
        \item $Q$ contains a subword $\exists\exists\forall$ or  $\exists\forall\exists$, and
        in this case there is a first-order $\tau$-sentence $\phi$ with quantifier prefix $Q$ such that $\HER(\phi)$ is $\coNP$-complete.
    \end{itemize}
\end{theorem}
\begin{proof}
    Clearly, both items describe  disjoint and complementary cases. The claim in the first item follows from 
    Corollary~\ref{cor:U*E*} and Theorem~\ref{thm:U*EU*}. 
    If $\tau$ contains a binary relation symbol, then the claim in the second item 
    follows from Corollary~\ref{cor:EEU}. 
    Otherwise, $\tau$ must contain a
    relation $R$ of arity at least three;
    however, we can use $R$ to model a binary relation, so the claim also holds in this case.
\end{proof}

Using Theorem~\ref{thm:Forb-TD} we can present a similar dichotomy to Theorem~\ref{thm:Q-classification}
even for the intersection of HerFO and CSPs.

\begin{corollary}\label{cor:Q-prefix-CSP}
    For every quantifier prefix $Q\in \{\exists,\forall\}^\ast$ one of the following statements hold:
    \begin{itemize}
        \item $Q$ is of the form $\forall^\ast\exists^\ast$, or of the form $\forall^\ast\exists\forall^\ast$, and in this case
        if $\CSP(\bA) = \HER(\phi)$ for some structure $\bA$ and some first-order sentence $\phi$ with quantifier prefix $Q$, then $\CSP(\bA)$
        is in $\PO$, or
        \item $Q$ contains a subword $\exists\exists\forall$ or  $\exists\forall\exists$, and
        in this case there is a structure $\bB$ such that $\CSP(\bB) = \HER(\phi)$ for some first-order sentence $\phi$ with 
        quantifier prefix $Q$, and $\CSP(\bB)$ is $\coNP$-complete.
    \end{itemize}
\end{corollary}

\section{Conclusion and Open Problems}
\label{sect:open}
We introduced the hereditary first-order model checking problem $\HerFO$,
and presented a complexity classification for $\HerFO$ based on allowed quantifier prefixes.
A number of open problems are left for future research. 
\begin{enumerate}
    \item We conjecture that there are first-order sentences $\phi$ such that 
    $\HER(\phi)$ is coNP-intermediate (assuming $\P\neq \NP$).
    \item Is every finite-domain CSP which is in $\HerFO$ also in P? Prove this without complexity-theoretic assumptions.
    \item Characterize the finite-domain CSPs in $\HerFO$.
    \item Is every CSP in $\HerFO$ also of the form $\HER(\phi)$ for some negative connected sentence $\phi$?
    \item It it true that the tractability problem for $\HerFO$ is undecidable even for first-order
    sentences with  quantifier prefix  $\exists\exists\forall$ (assuming $\P\neq\NP$)? Compare to Corollary~\ref{cor:undecidablity-fragments}
    and to the first item of Theorem~\ref{thm:Q-classification}. 
\end{enumerate}


\bibliographystyle{abbrv}
\bibliography{global.bib}

\appendix

\section{Examples}
\label{ap:examples}

\subsection{Polynomial-time solvable examples}
\label{ap:P-examples}
Example~\ref{ex:DAG} shows that $\CSP(\mathbb Q,<)$ is in  $\HerFO$ but not in FO.
There are also finite-domain CSPs that are in $\HerFO$, but not in $\FO$, as the following examples show. 
First, recall that the \textit{algebraic length} of an oriented path $P$ is the absolute value of the difference
between the number of forward edges and the number of backward edges in $P$. It is well-known (and straightforward to observe) that 
a digraph $D$ homomorphically maps to the directed path of length $2$ if and only if every oriented path
in $D$ has algebraic length at most $2$.

\begin{example}\label{ex:P2}
    The problem $\CSP(\vec{P_3})$ is in $\HerFO$. Consider a first-order formula $\phi$ saying that there
    is a loop, or there are (directed) edges $(x,y)$ and $(a,b)$ such that the following hold:
    \begin{itemize}[itemsep = 0.2pt]
        \item  $y$ has exactly one out-neighbour different from $b$, 
        \item $a$ has exactly one in-neighbour different from $x$, and
        \item every  $z\not\in\{x,y,a,b\}$ has exactly two out-neighbours different from $b$ and no in-neighbours, or
        $v$ has exactly two in-neighbours different from $x$ and no out-neighbours.
    \end{itemize}
    Suppose that a loopless digraph $\bD$ satisfies $\phi$.  It is not hard to observe that if $u$ and $v$ are
    witnessing vertices for $x$ and $b$, then there is an oriented path from $u$ to $v$ of algebraic length $3$. Hence, 
    if $\bD \not\in \HER(\lnot \phi)$, then either $\bD$ contains a loop or a path of algebraic length $3$, 
    and so $\bD \not\to \vec{P_3}$. Conversely, if $\bD\not\to \vec{P_3}$, then either $\bD$
    contains a loop or it has an
    oriented path of algebraic length $3$. In the former case, $\bD$ clearly satisfies $\phi$, and in the latter,     by choosing the shortest such path $v_1,\dots, v_n$ we find a substructure of $\bD$ that models $\phi$; namely,
    the substructure with vertex set $\{v_1,\dots, v_n\}$. Therefore, if $\bD\not\to \vec{P_3}$, then
    $\bD\not\in \HER(\lnot\phi)$. 
\end{example}

The following is an example of a CSP which is hereditarily definable by a
$\forall^\ast \exists \forall^\ast$-formula, and hence in P by Theorem~\ref{thm:U*EU*}.

\begin{example}\label{ex:cover-relation}
    For $n \ge 3$, denote by $\bA_n$ the oriented cycle obtained from the directed cycle $\vec{\bC}_n$ 
    be reversing the orientation of exactly one arc. We consider the $\CSP$ defined by homomorphically
    forbidding for every $n\ge 3$ the oriented cycle $\bA_n$ and the directed cycle $\vec{\bC}_n$. 
    Equivalently, a digraph 
    $\bD$ belongs to this $\CSP$ if and only if it is the cover relation of a partially ordered set.
    Consider the sentence
    \begin{align}
    \phi:= \exists x,y\forall z \exists a.~E(x,y) \land \big(E(z,a)~\lor~(z = x \land a\neq y \land E(x,a))~\lor
    (z = y)\big).\nonumber
    \end{align}
    If $\vec \bC_n$ or $\bA_n$, for some $n \geq 3$, homomorphically maps to  $\bD$,  
    then $\bD$ contains a substructure  satisfying $\phi$:
    if $\vec \bC_n$ homomorphically maps to $\bD$, 
    then in the image of the homomorphism, 
    there is an edge $(u,v)$ and every vertex $z$ has an out-neighbour $a$. 
    Otherwise, there is a homomorphism 
    from $\bA_n$ to $\bD$, 
    and $\bD$ contains an
     edge $(u,v)$ for which there
    is a directed $uv$-path $u, u_1,\dots, u_k, v$ where $k\ge 1$ (so $u_1\neq v$). 
    Notice that the digraph induced by $\{u,u_1,\dots, u_k,v\}$ in $\bD$ satisfies  $\phi$, 
    where $u$ and $v$ are witnesses for $x$ and $y$, respectively.

    Conversely, suppose that $\bD$ satisfies $\phi$. We will show that then there is a homomorphism 
    from $\vec \bC_n$ or from $\bA_n$, for $n \geq 3$, to $\bD$. This implies the statement, because if
    some substructure of $\bD$ satisfies $\phi$, there is a homomorphism from $\vec \bC_n$ or $\bA_n$ to this substructure,
    and hence in particular to $\bD$.
    Observe that if $\bD$ satisfies $\phi$ 
    there are $u,v\in D$ such that $(u,v)$ is an edge of $\bD$, and
    $u$ has an out-neighbour $w_1\neq v$. 
    Let $u,w_1,\dots, w_k$ be the
    largest directed path starting in $u$. If $w_k = v$, it follows that there is a homomorphism from
    the oriented cycle $\bA_{k+1}$ to $\bD$. Else, $w_k\neq v$, and since $\bD$ satisfies $\phi$, $w_k$ must
    have an out-neighbour $a$ and $a\in\{u,w_1,\dots, x_k\}$. Hence, there is a homomorphism from
    $\vec \bC_n$ to $\bD$ (and we can choose $n \geq 3$). 
    Therefore, the $\CSP$ considered in this example is hereditarily defined by the sentence $\lnot\phi$.
\end{example}

\subsection{HerFO and Datalog}

\emph{Datalog} can be seen as the subclass of SNP where we require that the first-order part $\psi$ of the SNP sentence 
\begin{itemize}
    \item 
is \emph{Horn}, i.e., written in conjunctive normal form such that each clause contains at most one positive literal, and 
\item is such that every positive literal in $\psi$ is existentially quantified.
\end{itemize}
A class ${\mathcal C}$ of finite models is \emph{in Datalog}
if there exists an SNP sentence $\Phi$ of the form described above such that a
finite structure is in  ${\mathcal C}$ if and only if it does \emph{not} satisfy
$\Phi$. Note  that if a class is in Datalog, then it is closed under homomorphisms\footnote{Our definition is standard, but different from the terminology of Feder and Vardi~\cite{FederVardi,federLICS}, who defined that $\CSP(\bB)$ is in Datalog if its complement is in Datalog in the standard sense.} and in P. 

A simple example of a CSP which is solved by a Datalog program is $\CSP(K_2)$. In Example~\ref{ex:CSP-K2} we  showed
that this class is not in $\HerFO$. 
There are also CSPs that are in $\HerFO$ and in P, but not in Datalog. 

\begin{example}\label{ex:and-or-scheduling}
For any positive integer $k$ consider a $(k+1)$-ary relation symbol $R$.
The problem $\CSP(\mathbb Q,\{(x,y_1,\dots, y_k)\colon x <\max\{y_1,\dots, y_k\})$ 
    is in P, but not in Datalog~\cite{ll}. 
    Notice that the problem is in $\HerFO$:
    consider the formula $$\exists x\forall y_1,\dots, y_k. \lnot R(x,y_1,\dots, y_k).$$
\end{example}

\begin{corollary}
    The classes of $\CSP$s in Datalog and in  $\PO\cap \HerFO$ are incomparable.
\end{corollary}

\subsection{HerFO and pp definitions}
A \emph{primitive positive formula} is an existential positive formula whose quantifier-free part
only uses variables and symbols from $\{\land, =\}\cup\tau$ (i.e., disjunction is forbidden). Given a structure $\bA$ we
say that a relation $R\subseteq A^r$
is \emph{primitively positively definable} (in $\bA$) if there is a primitive positive formula $\phi(x_1,\dots, x_r)$
such that $\bar a\in R$ if and only if $\bA\models\phi(\bar a)$. Primitive positive definitions are one
of the most elementary tools in constraint satisfaction theory. The following example (in combination with Example~\ref{ex:DAG})
shows that CSPs in $\HerFO$ are not preserved by primitive positive definitions:
specifically, they are not preserved by adding equality.

\begin{example}\label{ex:Q<=}
    The problem 
    $\CSP(\mathbb Q, <, =)$ is not hereditarily definable.  Consider the family of structures 
    $\bC_n$ and $\bD_n$ illustrated below (where undirected blue edges represent the (symmetric) binary
    relation $=$, and directed black edges represent the binary relation $<$). It is not hard to see that for every positive 
    integer $n$, 
    \begin{itemize}
        \item $\bC_n$ is a minimal obstruction of $\CSP(\mathbb Q, <, =)$, and 
        \item $\bD_n\in
    \CSP(\mathbb Q, <, =)$.
    \end{itemize}
    But again, for every $k \in {\mathbb N}$ one can choose  $n,m$ appropriately so that
    $\bC_n \equiv_k \bD_m$.
    \begin{center}
    \begin{tikzpicture}[scale = 0.8]

    \begin{scope}
        \node [vertex, label = {above:$v_1$}] (v1) at (-2.5,2) {};
        \node [vertex, label = {above:$v_2$}] (v2) at (-1,2) {};
        \node [vertex,  label = {above:$v_{n-1}$}] (v3) at (1,2) {};
        \node [vertex, label = {above:$v_n$}] (v4) at (2.5,2) {};
        \node [vertex, label = {below:$u_n$}] (u1) at (-2.5,-1.5) {};
        \node [vertex, label = {below:$u_{n-1}$}] (u2) at (-1,-1.5) {};
        \node [vertex, label = {below:$u_2$}] (u3) at (1,-1.5) {};
        \node [vertex, label = {below:$u_1$}] (u4) at (2.5,-1.5) {};
        \node [vertex, label = {left:$l_{n-1}$}] (l1) at (-2.5,1) {};
        \node [vertex, label = {left:$l_2$}] (l2) at (-2.5,-0.5) {};
        \node [vertex, label = {right:$r_2$}] (r1) at (2.5,1) {};
        \node [vertex, label = {right:$r_{n-1}$}] (r2) at (2.5,-0.5) {};
        
        \node  at (0,-2.75) {$\bC_n$};

        \foreach \from/\to in {v1/v2,  v3/v4, u4/u3, u2/u1}     
        \draw [arc] (\from) to (\to);
        \foreach \from/\to in {v1/l1, l2/u1, v4/r1, r2/u4}     
        \draw [edge, blue] (\from) to  (\to);
        \draw [thick, dotted] (-0.5,2) to  (0.5,2);
        \draw [thick, dotted] (-0.5,-1.5) to  (0.5,-1.5);
        \draw [thick, blue, dotted] (-2.5,0) to  (-2.5,0.5);
        \draw [thick, blue, dotted] (2.5,0) to  (2.5,0.5);
    \end{scope}
    
    \begin{scope}[xshift = 9cm]
        \node [vertex, label = {above:$v_1$}] (v1) at (-2.5,2) {};
        \node [vertex, label = {above:$v_2$}] (v2) at (-1,2) {};
        \node [vertex,  label = {above:$v_{n-1}$}] (v3) at (1,2) {};
        \node [vertex, label = {above:$v_n$}] (v4) at (2.5,2) {};
        \node [vertex, label = {below:$u_n$}] (u1) at (-2.5,-1.5) {};
        \node [vertex, label = {below:$u_{n-1}$}] (u2) at (-1,-1.5) {};
        \node [vertex, label = {below:$u_2$}] (u3) at (1,-1.5) {};
        \node [vertex, label = {below:$u_1$}] (u4) at (2.5,-1.5) {};
        \node [vertex, label = {left:$l_{n-1}$}] (l1) at (-2.5,1) {};
        \node [vertex, label = {left:$l_2$}] (l2) at (-2.5,-0.5) {};
        \node [vertex, label = {right:$r_2$}] (r1) at (2.5,1) {};
        \node [vertex, label = {right:$r_{n-1}$}] (r2) at (2.5,-0.5) {};
        
        \node  at (0,-2.75) {$\bD_n$};

        \foreach \from/\to in {v1/v2,  v3/v4, u3/u4, u1/u2}     
        \draw [arc] (\from) to (\to);
        \foreach \from/\to in {v1/l1, l2/u1, v4/r1, r2/u4}     
        \draw [edge, blue] (\from) to  (\to);
        \draw [thick, dotted] (-0.5,2) to  (0.5,2);
        \draw [thick, dotted] (-0.5,-1.5) to  (0.5,-1.5);
        \draw [thick, blue, dotted] (-2.5,0) to  (-2.5,0.5);
        \draw [thick, blue, dotted] (2.5,0) to  (2.5,0.5);
    \end{scope}

    \end{tikzpicture}
\end{center}
\end{example}

\subsection{The Henson digraphs}
\label{ap:Henson}

The \textit{Henson set} is the set ${\mathcal T}$ of tournaments $\mathbb T_n$ defined for positive
integer $n\ge 5$ as follows. The vertex set of $\mathbb T_n$ is $[n]$
and it contains the edges
\vspace{-2.4pt}
\begin{itemize}[itemsep = 0.8pt]
    \item $(1,n)$,
    \item $(i,i+1)$ for $i\in [n-1]$, and
    \item $(j,i)$ for $j > i+1$ and $(j,i)\neq (n,1)$.
\end{itemize}

It is straightforward to observe that, for any fixed (possibly infinite) set of tournaments $\calF$
(such as $\calT$), the class of oriented graphs
that  do not embed any tournament from $\calF$ 
is of the form $\CSP(\bB_{\mathcal F})$ for some countably infinite structure $\bB_{\mathcal F}$ (we may even choose $\bB_{\mathcal F}$ to be homogeneous). In the concrete situation of the class ${\mathcal T}$, $\CSP(\bB_{\mathcal T})$ is an example of $\coNP$-complete CSP (see, e.g.,~\cite[Proposition 13.3.1]{bodirsky2021}).
Moreover, it was shown in~\cite{BKR} that this CSP is in monadic second order logic. 
We strengthen this result by showing that $\CSP(\bB_\calT)$ is in $\HerFO$ --- and so, there is a monadic
second order sentence with only one existentially quantified unary predicate describing it (Observation~\ref{obs:HerFO-UMSO}).

We prove that the Henson set is first-order definable.  We use $\Cyc(x,y,z)$ as shorthand writing
for 
\[
\lnot E(x,x)\land \lnot E(y,y)\land \lnot E(z,z)\land E(x,y)\land E(y,z)\land E(z,x),
\]
i.e., $x,y,z$ is a directed $3$-cycle. We will also use standard set theoretic notation to 
simplify our writing, e.g., for variables $x,y_1,\dots, y_m$ we write
$x\in\{y_1,\dots, y_m\}$ instead of $\lor_{i\in[m]} x = y_i$, and similarly,
we write $|\{x_1,\dots, x_m\}| = k$ (resp.\ $\le k$) for the quantifier-free
formula stating that the cardinality of $\{x_1,\dots, x_m\}$ is $k$ (resp.\ at most $k$).
Consider the following first-order sentence.
\begin{align}
        \phi:=  \exists  x_1,&x_2,x_3,y_1,y_2,y_3 \forall z,a,b,c,d,e,f \exists g_1,g_2.\Cyc(x_1,x_2,x_2)\land \Cyc(y_1,y_2,y_3)\\ 
        \bigwedge \;& z\in\{x_1,x_2,y_2,y_3\} \lor \Cyc(y_3,z,x_1)\\
        \bigwedge \;& (z\in\{y_1,y_2\} \lor E(y_2,z))\land (z\in\{x_2,x_3\} \lor E(z,x_2)) \\
        \bigwedge  \;& \Cyc(a,b,c)\land \Cyc(a,b,d)\land \Cyc(a,b,e) \Rightarrow \big(|\{c,d,e\}|\le 2 \lor (x_1 = a\land y_3 = b)\big)\\
        \bigwedge  \;& \Cyc(a,b,c)\land \Cyc(a,b,d)\land \Cyc(b,c,e) \land \Cyc(c,a,f) \Rightarrow (|\{a,b,c,d,e,f\}|\le 5)\\
        \bigwedge  \;& \Cyc(a,b,c) \land \{x_1,y_3\}\not\subseteq \{a,b,c\} \Rightarrow \big[(\Cyc(a,b,g_1) \land \Cyc(b,c,g_2)) \nonumber\\
        \;&  ~~~\lor (\Cyc(b,c,g_1)\land \Cyc(c,a,g_2))\lor (\Cyc(c,a,g_1)\lor \Cyc(a,b,g_2))\big] 
 \end{align}

It is straightforward to translate $\phi$ to plain English, we do so in the following remark. We also depict
the first two lines of $\phi$ in Figure~\ref{fig:phiH}.

\begin{remark}\label{rmk:phi}
    A finite tournament $\mathbb T$ satisfies $\phi$ is and only if there are
    vertices $s_1,s_2,s_3$ (witnesses for $x_1,x_2,x_3$) and $t_1,t_2,t_3$ (witnesses for $y_1,y_2,y_3$)
    such that the following items hold.
    \begin{enumerate}
        \item $s_1,s_2,s_3$ and $t_1,t_2,t_3$ are directed cycles in $\mathbb T$.
        \item $t_3,t,s_1$ is a directed cycle for every $t\in T\setminus\{s_1,s_2,t_2,t_3\}$.
        \item There is an edge $E(t_2,t)$ and $E(s,s_2)$ for every $t\in T\setminus\{t_1,t_2\}$ and every $s\in T\setminus\{s_2,s_3\}$.
        \item Every edge $(s,t)\neq (s_1,t_3)$ belongs to at most $2$ directed $3$-cycles.
        \item At most $2$ edges of every directed $3$-cycle belong to $2$ different directed $3$-cycles.
        \item Every directed $3$-cycle that does not contain the edge $(s_1,t_3)$ contains exactly $2$ edges that
        belong to $2$ different directed $3$-cycles.
    \end{enumerate}
\end{remark}

\begin{figure}[ht!]
\centering
\begin{tikzpicture}[scale = 0.8]

  \begin{scope}
    \node [vertex, label = below:{$x_1$}] (x1) at (-4.5,1) {};
    \node [vertex, label = below:{$x_2$}] (x2) at (-3,1) {};
    \node [vertex, label = below:{$x_3$}] (x3) at (-1.5,1) {};
    \node [vertex, label = below:{$y_1$}] (y1) at (1.5,1) {};
    \node [vertex, label = below:{$y_2$}] (y2) at (3,1) {};
    \node [vertex, label = below:{$y_3$}] (y3) at (4.5,1) {};

    \node [vertex, fill = black, label = below:{$\scriptstyle z\not\in \{x_1,x_2,y_2,y_3\}$}] (z) at (0,-1) {};

    \foreach \from/\to in {x1/x2, x2/x3, y1/y2, y2/y3} 
    \draw [arc] (\from) to (\to);
    \foreach \from/\to in {y3/z, z/x1}   
    \draw [arc] (\from) to [bend left = 20] (\to);
    \foreach \from/\to in {y3/x3, y3/x2}   
    \draw [arc] (\from) to [bend right = 30] (\to);
    \foreach \from/\to in {y2/x3, y2/x2, y2/x1}   
    \draw [arc] (\from) to [bend right = 30] (\to);
    \foreach \from/\to in {y1/x3, y1/x2, y1/x1}   
    \draw [arc] (\from) to [bend right = 30] (\to);
    \draw [arc] (x1) to [bend left = 45] (y3);

  \end{scope}

\end{tikzpicture}
\caption{A partial depiction of the first three lines of $\phi$.}
\label{fig:phiH}
\end{figure}

\begin{lemma}\label{lem:Hensons->phi}
    If a tournament $\mathbb T$ belongs to the Henson set and has at least six vertices, 
    then $\mathbb T\models \phi$.
\end{lemma}
\begin{proof}
    Let $\mathbb T_n$ be a tournament in the Henson set for $n \ge 6$. Clearly, every directed cycle
    of $\mathbb T_n$ must contain some edge $(i,j)$ where $i < j$. It follows from the definition
    of $\mathbb T_n$ that the only such edges are $(1,n)$ and $(i,i+1)$ for $i\in [n-1]$. It should now
    be straightforward to notice that every directed cycle in $\mathbb T_n$ is of the form $n,i,1$ for some
    $i\in\{3,\dots, n-2\}$ or $i-1,i,i+1$ for $i\in\{2,\dots, n-1\}$. It thus follows that every edge
    other than $(1,n)$ belongs to at most $2$ different directed $3$-cycles, and that at most two edges
    of every directed $3$-cycle  belongs to at most $2$ different directed $3$-cycles. Moreover, 
    every $3$-cycle not containing $(1,n)$ has exactly two different edges belonging to two different directed
    cycles. Therefore, by evaluating $(x_1,x_2,x_3,y_1,y_2,y_3)\mapsto (1,2,3,n-2,n-1,n)$ we conclude
    that $\mathbb T_n\models \phi$.
\end{proof}

Let $\psi$ be any $\exists\forall$-formula axiomatizing the Henson tournament on $5$ vertices,
and let $\phi_{\calT}: = \psi\lor \phi$.

\begin{lemma}\label{lem:Henson-ax}
    A finite tournament $\mathbb T$ belongs to the Henson set if and only if
    $\mathbb T\models \phi_\calT$. 
\end{lemma}
\begin{proof}
    It follows from Lemma~\ref{lem:Hensons->phi} and the definition of $\phi_\calT$, that
    if $\mathbb T$ is a tournament in the Henson set, then $\mathbb T\models\phi_\calT$.
    Suppose that a tournament $\mathbb T$ satisfies $\phi_\calT$. Notice that $\mathbb T$
    must have at least five vertices, and if $\mathbb T$ has exactly five vertices, then
    $\mathbb T$ is the Henson set on five vertices. Now suppose that $|T| \ge 6$ and
    thus $\mathbb T\models \phi$. Let $(s_1,s_2,s_3,t_1,t_2,t_3)$ be a wining move
    for existential player. Let $\bC_0$ be the directed $3$-cycle with vertex set $s_1,s_2,s_3$,
    let $v_1$ be a vertex such that $\Cyc(s_2,s_3,v_1)$ (such a vertex exists because $(s_1,s_2)$
    only belongs to $C_0$, and $\bC_0$ must contain two edges that belong to two different
    triangles). Let $\bC_1$ be the directed $3$-cycle with vertex set $s_2,s_3,v_1$. For
    $n\ge 2$, if there is a directed $3$-cycle $\bC$ different from $\bC_{n-2}$ and having a
    common edge with $\bC_{n-1}$, let $\bC_n = \bC$; otherwise, let $\bC_n = \bC_{n-1}$.
    Notice that if $t_3\not\in\{C_{n-1}\}$, then $\bC_n\neq \bC_{n-1}$ (follows from
    line (6) of definition of $\phi$, equiv.\ from the sixth item of Remark~\ref{rmk:phi}).
    If there is some integer $n$ such that $t_3\in C_n$, let $M$ be the minimum  such that
    $t_3\in C_M$ otherwise, let $M = \omega$. So, for every $n\le  M$ there is exactly one vertex
    $v_n\in C_n\setminus\{C_{n-1}\}$. We prove that $C_{n+1} = \{v_{n-1},v_n,v_{n+1}\}$ 
    for $2\le n <M$, and proceeding by contradiction, let $N+1$ be a minimum counterexample. 
    It follows by finite induction that for $n \le N$ the directed cycle $\bC_n$ has vertex
    set $\{v_{n-2},v_{n-1},v_n\}$ and edges $(v_{n-2},v_{n-1}),(v_{n-1},v_n),(v_n,v_{n-2})$ ---
    it can be convenient to have the following picture in mind.\\
    
    \begin{center}
    \begin{tikzpicture}[scale = 0.8]

  \begin{scope}
    \node [vertex, label = below:{$s_1$}] (x1) at (-6,1) {};
    \node [vertex, label = below:{$s_2$}] (x2) at (-4.5,1) {};
    \node [vertex, label = below:{$s_3$}] (x3) at (-3,1) {};
    \node [vertex, label = below:{$v_1$}] (v1) at (-1.5,1) {};

     \node [vertex, label = below:{$v_{N-3}$}] (s) at (1.5,1) {};
    \node [vertex, label = below:{$v_{N-2}$}] (u) at (3,1) {};
    \node [vertex, label = below:{$v_{N-1}$}] (v) at (4.5,1) {};
    \node [vertex, label = below:{$v_N$}] (w) at (6,1) {};

    \node [vertex, label = below:{$v_{N+1}$}] (C) at (4.5,0) {};
      
    \foreach \from/\to in {x1/x2, x2/x3, x3/v1, u/v, v/w, u/C, C/w, s/u}
    \draw [arc] (\from) to (\to);
    \foreach \from/\to in {x3/x1, v1/x2, w/u, v/s, w/s}   
    \draw [arc] (\from) to [bend right] (\to);
    \draw [dotted] (-1,1) to (1,1);
    
  \end{scope}  
\end{tikzpicture}
\end{center}

Since $E(v_N,s_2)$ (from second line of definition of $\phi$) and every edge other than
$(s_1,t_3)$ belongs to at most $2$ triangles, it follows by finite induction that 
$E(v_N,v_{N_3})$. With similar arguments we see that $E(v_{N+1},s_2)$ and $E(v_{N+1},v_{N-3})$. 
The latter implies that the edge $(v_{N-3},v_{N-2})$ belongs to three different directed triangles, 
contradicting the fact that $\mathbb T$ satisfies $\phi$. To conclude the proof let $m$ be the
minimum integer such that $v_n\in \{s_1,s_2,s_3,v_1,\dots, v_{n-1}, t_1,t_2,t_3\}$. If $m = 1$,
then $v_1\in\{t_1,t_2,t_3\}$, and since $E(s_3,t_1)$ it follows that $v_1 = t_1$, so 
$\{s_1,s_2,s_3,t_1,t_2,t_3\}$ induces a Henson tournament on $6$ vertices. Otherwise, 
if $m > 1$, it follows similarly as before that $E(v_m,s_1)$, $E(v_m,s_2)$, $E(v_m,s_3)$ and
$E(v_m,v_i)$ for $i< m-1$.  In particular, $v_m\not\in\{s_1,s_2,s_3,v_1,\dots, v_{m-1},t_3\}$, 
and since $E(t_3,v_m)$ then $v_m\neq t_2$, and thus $v_m = t_1$. Again, it is straightforward
to observe that $\mathbb T[\{s_1,s_2,s_3,v_1,\dots, v_{m-1},t_1,t_2,t_3\}]$ induces
a tournament of the Henson set. Finally, we show that $T = \{s_1,s_2,s_3,v_1,\dots, v_{m-1},t_1,t_2,t_3\}$. 
On the contrary, suppose that there is some $t\in T\setminus\{s_1,s_2,s_3,v_1,\dots, v_{m-1},t_1,t_2,t_3\}$. 
By the second line of the definition of $\phi$ there are edges $(t,s_1)$ and $(t_3,t)$. Let
$v$ be the maximum vertex with respect to the ordering $s_1 \le s_2 \le s_3 \le v_1 \le \dots \le t_3$
such that $E(t,v)$. Hence, if $u$ is the successor of $v$ (according to the previous linear ordering), 
then $\Cyc(v,u,t)$, and thus the edge $E(v,u)$ belongs to three different directed triangles,
contradicting the definition of $\phi$. Therefore, $\mathbb T$ is tournament in the Henson set. 
\end{proof}

The following theorem is a direct consequence of the previous lemma, because the class
of tournaments is universally definable.

\begin{theorem}
    The Henson set is first-order definable. In particular, it is definable
    by an $\exists^6\forall^6\exists^2$-formula.
\end{theorem}
\begin{proof}
    Let $\varphi$ be any universal formula that axiomatizes the class of tournaments. 
    Then, $\varphi\land \phi_\calT$ axiomatizes the Henson set (Lemma~\ref{lem:Henson-ax}).
\end{proof}

Notice that if $\phi$ is the first-order sentence defining the Henson set $\mathcal T$, then 
an oriented graph $G$ is $\mathcal T$-free if and only if every substructure of $G$
is a loopless oriented graph that satisfies $\lnot\phi$. Thus, the following statement
is implied by the previous theorem.

\begin{corollary}\label{thm:HensonCSP}
    If $\calT$ is the Henson set, then the class of $\calT$-free oriented graphs is a 
    $\coNP$-complete $\CSP$ in $\HerFO$. 
\end{corollary}

\end{document}